\documentclass[10pt,a4paper]{article}
\usepackage{a4wide}
\usepackage{amsfonts}
\usepackage{amscd}
\usepackage{amssymb}
\usepackage{amsthm}
\usepackage{amsmath}
\usepackage{latexsym}
\usepackage{dsfont}
\usepackage{hyperref}
\usepackage[usenames,dvipsnames]{color}
\usepackage{esvect} %Vector arrows
\usepackage{algorithm}
\usepackage{algpseudocode}
\usepackage{enumerate}
\usepackage[dvipsnames]{xcolor}
\usepackage{bm}
\usepackage{subfigure}
\usepackage{graphicx}

\usepackage{caption}
\usepackage{multirow}

\theoremstyle{plain}
\newtheorem{thm}{Theorem}[section]
\theoremstyle{plain}
\newtheorem{lem}[thm]{Lemma}
\newtheorem{prop}[thm]{Proposition}
\newtheorem{cor}[thm]{Corollary}
\theoremstyle{definition}
\newtheorem{defi}[thm]{Definition}
\newtheorem{rem}[thm]{Remark}
\newtheorem{assumption}[thm]{Assumption}

\newcommand{\ga}{\alpha}
\newcommand{\gb}{\beta}
\newcommand{\gd}{\delta}
\newcommand{\eps}{\ensuremath{\varepsilon}}

\renewcommand{\gg}{\gamma}
\newcommand{\gk}{\kappa}
\newcommand{\gl}{\lambda}
\newcommand{\go}{\omega}
\newcommand{\gs}{\sigma}
\newcommand{\gt}{\theta}

\newcommand{\gO}{\Omega}
\newcommand{\gS}{\Sigma}

\newcommand{\cA}{\mathcal{A}}
\newcommand{\cB}{\mathcal{B}}
\newcommand{\cC}{\mathcal{C}}
\newcommand{\cD}{\mathcal{D}}
\newcommand{\cE}{\mathcal{E}}

\newcommand{\cG}{\mathcal{G}}

\newcommand{\cI}{\mathcal{I}}

\newcommand{\cK}{\mathcal{K}}
\newcommand{\cL}{\mathcal{L}}

\newcommand{\cN}{\mathcal{N}}
\newcommand{\cO}{\mathcal{O}}
\newcommand{\cP}{\mathcal{P}}
\newcommand{\cQ}{\mathcal{Q}}

\newcommand{\cV}{\mathcal{V}}

\newcommand{\cX}{\mathcal{X}}

\newcommand{\bE}{\mathbb{E}}

\newcommand{\bN}{\mathbb{N}}

\newcommand{\bP}{\mathbb{P}}
\newcommand{\bQ}{\mathbb{Q}}
\newcommand{\bR}{\mathbb{R}}

\newcommand{\bT}{\mathbb{T}}

\newcommand{\bZ}{\mathbb{Z}}
\newcommand{\rC}{{\rm C}}
\newcommand{\rD}{{\rm D}}

\newcommand{\rS}{{\rm S}}

\newcommand{\mfH}{\mathfrak{H}}
\newcommand{\mfI}{\mathfrak{I}}

\newcommand{\mfT}{\mathfrak{T}}

\newcommand*{\lrscript}[5]{{\vphantom{#1}}_{#2}^{#3}{#1}_{#4}^{#5}}
\newcommand{\dualpair}[4]{\ensuremath{\lrscript{\langle}{#1}{}{}{} #3 ,#4 \rangle_{#2}}}
\newcommand{\be}{\begin {equation}}
\newcommand{\ee}{\end  {equation}}
\newcommand{\bee}{\begin {equation*}}
\newcommand{\eee}{\end {equation*}}
\newcommand{\ol}{\overline}

\newcommand{\floor}[1]{\lfloor #1 \rfloor}
\newcommand{\ceil}[1]{\lceil #1 \rceil}

\newcommand{\indi}{\mathds{1}}

\DeclareMathOperator*{\essinf}{\text{ess\,inf}}
\DeclareMathOperator*{\esssup}{\text{ess\,sup}}

\newcommand{\MC}{Monte Carlo}

\newcommand{\ML}{multilevel\ }
\newcommand{\MML}{Multilevel\ }

\newcommand{\MCMC}{Markov chain \MC}

\begin{document}
	
	\title{
		Multilevel Markov Chain \MC \\ for Bayesian Elliptic Inverse Problems \\ with Besov Random Tree Priors}

	\date{}
	
	\author{
		Andreas Stein 
		\footnote{Seminar for Applied Mathematics, Department of Mathematics, ETH Zürich}  \footnote{Corresponding author. Email: andreas.stein@sam.math.ethz.ch}
		\footnote {The datasets generated during and/or analysed during the current study are available from the corresponding author upon reasonable request.}
		\and
		Viet Ha Hoang 
		\footnote{Division of Mathematical Sciences, School of Physical and Mathematical Sciences, Nanyang Technological University Singapore}	
	}
	
	\maketitle
	
	\begin{abstract}
		We propose a \ML \MCMC\,-FEM algorithm to solve elliptic Bayesian inverse problems with "Besov random tree prior". These priors are given by a wavelet series with stochastic coefficients, and certain terms in the expansion vanishing at random, according to the law of so-called \emph{Galton-Watson trees}.
		This allows to incorporate random fractal structures and large deviations in the log-diffusion, which occur naturally in many applications from geophysics or medical imaging.  
		This framework entails two main difficulties: First, the associated diffusion coefficient does not satisfy a uniform ellipticity condition, which leads to non-integrable terms and thus divergence of standard \ML estimators. 
		Secondly, the associated space of parameters is Polish, but not a normed linear space.
		We address the first point by introducing cut-off functions in the estimator to compensate for the non-integrable terms, while the second issue is resolved by employing an independence Metropolis-Hastings sampler. 
		The resulting algorithm converges in the mean-square sense with essentially optimal asymptotic complexity, and dimension-independent acceptance probabilities.
	\end{abstract}
	
	%%%%%%%%%%%%%%%%%%%%%%%%%%%%%%%%%%%%%%%%%%%%%%%%%%%%%%%%%
	\section{Introduction}
	\label{sec:intro}

	%%%%%%%%%%%%%%%%%%%%%%%%%%%%%%%%%%%%%%%%%%%%%%%%%%%%%%%%%
	
	Countless phenomena in the natural sciences and engineering are modeled by partial differential equations (PDEs). Parameters in the corresponding models are in general subject to uncertainty, due to incomplete information, measurement errors, etc. Therefore, the PDE parameters are often considered as random variables, or (possibly) infinite-dimensional random fields.
	A well-studied example are second-order elliptic equations with a random diffusion coefficient as statistical model for uncertain permeability/conductivity in a given physical domain.
	In many applications it is then of interest to solve the associated \emph{inverse problem}, that is, to infer realizations of the model parameter based on discrete observations of the solution to the PDE model. 
	Important applications are electrical resistivity tomography in geophysical engineering (\cite{galetti2018transdimensional}) or electromyography for medical applications (\cite{roerich2021bayesian}). 
	In any case, the inverse problem is ill-posed and requires appropriate regularization techniques.
	
	A popular approach is to consider the inverse problem from a statistical or \emph{Bayesian} perspective (\cite{kaipio2006statistical, calvetti2007introduction, stuart2010inverse, DS17}) with its solution given a by probability measure on a suitable space of parameters. This so-called \emph{posterior} measure is inferred by conditioning an a-priori chosen \emph{prior} measure on the observed data. 
	Well-posedness of the Bayesian inverse problem (BIP) is ensured under mild assumptions, and a-priori model information may be incorporated by selecting an appropriate \emph{prior} model for the parameter space.
	In the wake of the pioneering work of Stuart \cite{stuart2010inverse}, there has been an explosion of interest in BIPs and inverse uncertainty quantification in the past decade, see e.g. \cite{cotter2009bayesian, hoang2012bayesian, DHS12, schillings2013sparse, hosseini2017well, latz2020well, monard2021statistical}. 
	
	From a computational viewpoint, solving the inverse problem amounts to sampling from a conditional probability measure, which is known only up to a normalization constant. 
	Hence, \MCMC\,(MCMC) methods are used extensively in Bayesian inference, see for instance \cite{cotter2013mcmc, DS17, rudolf2018generalization, latz2021generalized}. 
	These acceptance-rejection algorithms rely on forward solves of the corresponding (PDE) model, that involve discretization errors (for instance due to finite element approximations) and possibly come at high computational costs. 
	These issues have been addressed by the development of multilevel \MC\, algorithms for BIPs, a non-exhaustive list includes \cite{hoang2013complexity, scheichl2017quasi, latz2018multilevel, hoang2020analysis, dodwell2019multilevel, madrigal2021analysis}. Multilevel \MCMC\, (ML-MCMC) methods reduce the complexity to compute quantities of interest with respect to the Bayesian posterior by orders of magnitude, when compared to their "standard" MCMC counterparts.  
	However, a drawback of many ML-MCMC approaches for elliptic BIPs is that they require a \emph {uniform ellipticity} condition on the random diffusion coefficient. This requirement excludes the important log-Gaussian prior, let alone models with heavier tails such as Besov priors~\cite{LassasBesov09}. 
	To the best of our knowledge, this issue has only been fully addressed in \cite{hoang2020analysis, hoang2021multilevel} for elliptic resp. parabolic BIPs with Gaussian prior. 

	Unfortunately, Gaussian prior models are not able to capture large deviations, due to their fast decaying tails. Moreover, Gaussian or Besov priors can not incorporate fractal (spatial) structures in the posterior model, which occur naturally in subsurface flow or medical imaging applications. For this reason, \emph{Besov random tree priors} have recently been introduced in~\cite{KLSS21} for linear inverse problems, and have been proposed as log-diffusion coefficient in a random elliptic PDE model in \cite{SS22}.
	These priors are given by a wavelet series with stochastic coefficients, and certain terms in the expansion vanishing at random, according to the law of so-called \emph{Galton-Watson trees}.
	Samples of the corresponding random fields involve fractal geometries, hence the Besov random tree prior may be a viable candidate in applications, 
	where models based on Gaussian random fields do not allow for sufficient flexibility. 
	The degree and Hausdorff dimension of the fractal structures are controlled by a steering parameter $\gb\in[0,1]$, the so-called \emph{wavelet-density}.

	%%%%%%%%%%%%%%%%%%%%%%%%%%%%%%%%%%%%%%%%%%%%%%%%%%%%%%%%%
	\subsection{Contributions}
	\label{sec:Contr}
	%%%%%%%%%%%%%%%%%%%%%%%%%%%%%%%%%%%%%%%%%%%%%%%%%%%%%%%%%
	
	We develop a ML-MCMC-finite element sampling algorithm for elliptic BIPs with Besov random tree prior. The results build on and complement the analysis of the corresponding companion paper~\cite{SS22} on the elliptic forward problem with Besov random tree coefficient.
	The hyper-parameters of the algorithm are tuned with respect to the regularity of the corresponding forward problem and we 
	provide an error-vs-work analysis for the ML-MCMC algorithm. 
	Our complexity estimates show that the proposed approach has essentially the same computational complexity as the forward MLMC method from~\cite{SS22} (up to logarithmic terms), and is therefore asymptotically optimal.   
	The results hold in particular for "standard" Besov priors on the torus with wavelet density $\gb=1$. 
	We emphasize that \emph{no uniform-ellipticity assumptions} are necessary in the forward model, as our ML-MCMC estimator compensates for non-integrable terms in the Bayesian potential without introducing an additional bias. In contrast, failing to take into account the unboundedness of the solution to the forward equation  and the Bayesian potential would result in highly inaccurate results (see e.g. the numerical experiments in \cite{hoang2020analysis}).
	We further use an independence Metropolis-Hastings sampler, hence the algorithm may be applied to general (non-linear) parameter spaces, such as the Polish space of GW trees. 
	While we restrict our analysis to Besov random tree priors in this article, it is straightforward to apply the presented algorithm to different prior models associated to a non-normed parameter space without uniform-ellipticity condition.

	%%%%%%%%%%%%%%%%%%%%%%%%%%%%%%%%%%%%%%%%%%%%%%%%%%%%%%%%%
	\subsection{Layout of this paper}
	\label{sec:Layout}
	%%%%%%%%%%%%%%%%%%%%%%%%%%%%%%%%%%%%%%%%%%%%%%%%%%%%%%%%%
	We fix the basic notation for this article in Section~\ref{sec:Notat}.
	Section~\ref{sec:BIP} introduces general elliptic BIPs and establishes results on well-posedness and data-dependence of the posterior measure.
	We introduce the Besov random tree priors in Section~\ref{sec:besov-rv}, where we also recall well-posedness and pathwise approximation results of the associated elliptic forward problem from~\cite{SS22} for the reader's convenience.
	Section~\ref{sec:BIP-random-tree} introduces the BIP with Besov random tree prior and the combined dimension truncation and finite element approximation of the posterior measure. We further prove a-priori error estimates on the posterior approximation in the Hellinger distance.  
	In Section~\ref{sec:MCMC} we introduce our ML-MCMC algorithm, prove convergence of the root-mean-squared error and provide the corresponding error-vs-work analysis for the entire range of regularity parameters in the prior model. 
	We validate our theoretical findings by several numerical experiments in Section~\ref{sec:numerics}.
	
	%%%%%%%%%%%%%%%%%%%%%%%%%%%%%%%%%%%%%%%%%%%%%%%%%%%%%%%%%
	\subsection{Notations}
	\label{sec:Notat}
	%%%%%%%%%%%%%%%%%%%%%%%%%%%%%%%%%%%%%%%%%%%%%%%%%%%%%%%%%
	
	We denote by $\cV'$ the topological dual for any vector space 
	$\cV$ and by $\dualpair{\cV'}{\cV}{\cdot}{\cdot}$ the associated dual pairing.
	For any metric space $(\cX,d_\cX)$ we denote by $B_\gl(y):=\{x\in\cX|\,d_\cX(x,y)\le \gl\}\subset\cX$ the closed ball with radius $\gl>0$ around $y\in\cX$. 
	If $d_\cX$ is induced by a norm $\left\|\cdot\right\|_\cX$ via $d_\cX(x,y)=\|x-y\|_\cX$, we write 
	$(\cX, \left\|\cdot\right\|_\cX)$ for the corresponding normed space.
	If $\cX=\bR^k$ for a $k\in\bN$, we use the Euclidean metric, unless stated otherwise.

	The Borel $\gs$-algebra of any metric space $\cX$ is generated by the open sets in $\cX$ and denoted by $\cB(\cX)$.
	For any $\gs$-finite and complete measure space $(E,\cE,\mu)$, 
	a Banach space $(\cX, \left\|\cdot\right\|_\cX)$, 
	and 
	integrability exponent $p\in[1,\infty]$,
	we define the Lebesgue-Bochner spaces 
	\begin{equation*}
		L^p(E, \mu; \cX):=\{\varphi:E\to\cX|\;
		\text{$\varphi$ is strongly measurable and $\|\varphi\|_{L^p(E, \mu; \cX)}<\infty$}  \},
	\end{equation*}
	where
	\begin{equation*}
		\|\varphi\|_{L^p(E, \mu; \cX)}:=
		\begin{cases}
			\left(\int_{E}\|\varphi(x)\|_\cX^p\mu(dx)\right)^{1/p},\quad &p\in[1,\infty) \\
			\esssup\limits_{x\in E} \|\varphi(x)\|_\cX,\quad &p=\infty.
		\end{cases}
	\end{equation*}
	In case that $\cX=\bR$, we use the shorthand notation $L^p(E,\mu):=L^p(E,\mu;\bR)$.
	If $E\subset\bR^d$ is a subset of Euclidean space, 
	we assume $\cE=\cB(E)$ and $\mu$ is the Lebesgue measure, 
	and write $L^p(E):=L^p(E,\mu;\bR)$,
	unless stated otherwise.

	For a probability space $(\gO,\cA,\bP)$ and a Banach space-valued random variable $X:\gO\to\cX$, 
	we denote by $\bE_\bP(X)=\int_\gO X(\go)d\bP(\go)$ the expectation of $X$ with respect to $\bP$.
	For any two measures $\bQ_1, \bQ_2$ on $(\gO,\cA)$, that are both absolutely continuous with respect to a reference measure $\bQ_0$ on $(\gO,\cA)$, the \emph{Hellinger distance} of $\bQ_1$ and $\bQ_2$ is given by 
	\begin{equation*}
		d_{\rm Hell}(\bQ_1, \bQ_2):=
		\left(
		\frac{1}{2}\int_\gO \left(\sqrt{\frac{d\bQ_1}{d\bQ_0}(\go)}-\sqrt{\frac{d\bQ_2}{d\bQ_0}(\go)}\right)^2
		d\bQ_0(\go)
		\right)^{1/2}.
	\end{equation*}

	For any bounded and connected spatial domain $\cD\subset \bR^d$ we denote for $k\in\bN$ and $p\in[1,\infty]$ 
	the standard Sobolev space $W^{k,p}(\cD)$ with $k$-order weak derivatives in $L^p(\cD)$.
	The Sobolev-Slobodeckji space with fractional order $s\ge0$ is denoted by $W^{s,p}(\cD)$. 
	Furthermore, $H^s(\cD):=W^{s,2}(\cD)$ for any $s\ge0$ and we use the identification $H^0(\cD)=L^2(\cD)$. 
	Given that $\cD$ is a Lipschitz domain, we define for any $s>1/2$
	\begin{equation}
		H_0^s(\cD):={\rm ker}(\gamma_0) = \{\varphi\in H^s(\cD)|\; \gamma_0(\varphi)=0 \;\mbox{on}\; \partial\cD \},
	\end{equation}
	Here, 
	$\gamma_0\in \cL(H^s(\cD),H^{s-1/2}(\partial \cD))$ denotes the trace operator.
	
	Let $\rC(\ol\cD)$ denote the space of all continuous functions $\varphi:\ol\cD\to\bR$.
	For any $\ga\in\bN$, $\rm \rC^{\ga}(\ol\cD)$ is the space of all functions $\varphi\in\rC(\ol\cD)$ 
	with $\ga$ continuous partial derivatives. 
	For non-integer $\ga>0$,
	we denote by $\rC^{\ga}(\ol\cD)$ the space of all 
	$\varphi\in \rC^{\floor{\ga}}(\ol\cD)$ with $\ga-\floor{\ga}$-Hölder 
	continuous $\floor{\ga}$-th partial derivatives.
	For any positive, real $\ga>0$ we further denote by $\cC^\ga(\cD)$ the \emph{Hölder-Zygmund space} of smoothness $\ga$. 
	We refer to, e.g., \cite[Section 1.2.2]{TriebelTOFS2} for a definition.
	We denote by $\rS(\bR^d)$ the Schwarz space of all smooth, rapidly decaying functions, 
	and with $\rS'(\bR^d)$ its dual, the space of tempered distributions.
	Moreover, for any open set $O\subseteq\bR^d$, 
	$\mathrm D(O)$ denotes the space of all smooth functions $\varphi\in \rC^\infty(O)$ with compact support in $O$.
	
	For the finite element error analysis we introduce a countable set $\mfH\subset(0,\infty)$, 
	and denote by $h\in \mfH$ a generic finite element refinement parameter. 
	We further assume the existence of a strictly decreasing sequence 
	$(h_\ell, \ell\in\bN)\subset\mfH$ such that $\lim_{\ell\to\infty} h_\ell = 0$.

	%%%%%%%%%%%%%%%%%%%%%%%%%%%%%%%%%%%%%%%%%%%%%%%%%%%%%%%%%
	\section{Bayesian Elliptic Inverse Problems}
	\label{sec:BIP}
	%%%%%%%%%%%%%%%%%%%%%%%%%%%%%%%%%%%%%%%%%%%%%%%%%%%%%%%%%
	
	%%%%%%%%%%%%%%%%%%%%%%%%%%%%%%%%%%%%%%%%%%%%%%%%%%%%%%%%%
	\subsection{Forward PDE model}
	\label{sec:forward-PDE}
	%%%%%%%%%%%%%%%%%%%%%%%%%%%%%%%%%%%%%%%%%%%%%%%%%%%%%%%%%
	
	Let $(\gO, \cA, \bP)$ be a complete probability space (of parameters $\go\in\gO$), and let $\cD\subset\bR^d$, $d\in\{1,2,3\}$ be a convex polygonal domain, 
	with the boundary $\partial \cD$ consisting of a finite number of line or plane segments.
	We consider the random (or "parametric") elliptic problem to find $u(\go):\cD\to \bR$ for given $\go\in\gO$ such that 
	\begin{equation}\label{eq:ellipticpde}
		\begin{alignedat}{2}
			-\nabla\cdot(a(\go)\nabla u(\go)) &= f\quad &&\text{in $\cD$}, \quad 
			u(\go) = 0 \quad &&\text{on $\partial\cD$}.
		\end{alignedat}
	\end{equation}
	The diffusion coefficient $a:\gO\to L^\infty(\cD)$ in Problem~\eqref{eq:ellipticpde} is a suitable random field and the source term $f:\cD\to\bR$ is assumed to be a deterministic function for the sake of simplicity.
	For the variational formulation of Problem~\eqref{eq:ellipticpde} 
	we define $H:=L^2(\cD)$, $V:=H_0^1(\cD)$ and recall that 
	$\left\|\cdot\right\|_V:V\to\bR_{\ge 0},\: v\mapsto \|\nabla v\|_H$ defines a norm on $V$ by Poincare's inequality.
	For fixed $\go\in\gO$,  we call $u(\go)\in V$ a \textit{pathwise weak solution} to Problem~\eqref{eq:ellipticpde} if for any $v\in V$ it holds 
	\begin{equation}\label{eq:ellipticpdeweak}
		\int_\cD a(\go)\nabla u(\go)\cdot\nabla v dx = \dualpair{V'}{V}{f}{v}.
	\end{equation}
	To ensure existence and uniqueness of pathwise weak solutions we assume $f\in V'$, and that $a:\gO\to L^\infty(\cD)$ is strongly $\cA/\cB(L^\infty(\cD))$-measurable such that  
	\begin{equation}\label{eq:coervity}
		a_-(\go):=\essinf_{x\in\cD} \: a(x,\go)>0,\quad \text{$P$-a.s.}
	\end{equation}
	It is then a standard result (see, e.g. \cite[Theorem 3.2]{SS22}) to show that the \textit{parameter-to-solution map} $u:\gO\to V$ is well-defined and (strongly) $\cA/\cB(V)$-measurable.
	
	%%%%%%%%%%%%%%%%%%%%%%%%%%%%%%%%%%%%%%%%%%%%%%%%%%%%%%%%%
	\subsection{Bayesian inverse problem}
	\label{sec:BIP-PDE}
	%%%%%%%%%%%%%%%%%%%%%%%%%%%%%%%%%%%%%%%%%%%%%%%%%%%%%%%%%	
	
	To introduce the inverse problem, we consider the \textit{parameter-to-observation map}
	\begin{equation}\label{eq:observation}
		\cG:\gO\to \bR^k,\quad \go\mapsto [\cO \circ u](\go)  
	\end{equation}
	%with observation functional $\cO\in(V')^k$ for $k\in\bN$.
	with bounded linear observation functional $\cO\in((H_0^{\theta_\cO}(\cD))')^k$ for $\theta_\cO\in(1/2,1]$ and $k\in\bN$.
	We assume noisy observations $\gd$ of the form
	\begin{equation}\label{eq:noisy_obs}
		\gd = \cG(\go) + \vartheta,
	\end{equation}
	where $\vartheta$ is centered \emph{Gaussian observation noise} on $\bR^k$.
	Thus, $\vartheta$ is distributed with density  
	\begin{equation}\label{eq:Gauss-noise}
		\rho(x)=(2\pi\det(\gS))^{-\frac{k}{2}}\exp\left( -\frac{1}{2}x^\top\gS^{-1}x\right),
		\quad x\in\bR^k,
	\end{equation}
	for a symmetric and positive definite covariance matrix $\gS\in\bR^{k\times k}$. 
	
	Given an observation $\gd$, we aim to derive the \emph{Bayesian posterior probability measure} $\bP_\gd:=\bP(\cdot\,|\,\gd)$ on $(\gO, \cA)$.
	Note that for given $\go\in\gO$, the distribution of $\gd=\cG(\go)+\vartheta$ (conditional on a given realization $\cG(\go)$) is $\bP$-.a.s. absolutely continuous with respect to $\cN(0, \gS)$, hence $\bP_\gd$ is given by \emph{Bayes' Theorem}: 
	\begin{prop}
		The posterior measure $\bP_\gd$ is absolutely continuous with respect to the prior measure $\bP$, with Radon-Nikodym derivative given by
		\begin{equation}\label{eq:RN_derivative}
			\frac{d\bP_\gd}{d\bP}(\go)
			=\frac{\rho(\gd-\cG(\go))}{\int_\gO \rho(\gd-\cG(\go))d\bP(\go)}
			= \frac{\exp\left(-\Phi(\go; \gd)\right)}
			{Z(\gd)}.
		\end{equation}
		In~\eqref{eq:RN_derivative}, we have defined the \emph{Bayesian potential} $\Phi:\gO\times\bR^k\to\bR$ as
		\begin{equation}\label{eq:potential}
			\Phi(\go; \gd)
			:=-\log\left(\rho(\gd-\cG(\go))\right)
			=\frac{k\log\left(2\pi\det(\gS)\right)}{2} 
			+ \frac{1}{2}(\gd-\cG(\go))^\top\gS^{-1}(\gd-\cG(\go)),
		\end{equation}	
		and the normalizing constant $Z(\gd)>0$ is given by
		\begin{equation}\label{eq:norm-constant}
			Z(\gd):=\int_\gO\exp\left(-\Phi(\go; \gd)\right)d\bP(\go).
		\end{equation}
	\end{prop}
	\begin{proof}
		The observation map $\cG:\gO\to \bR^k$ in~\eqref{eq:observation} inherits the measurability from $u:\gO\to V$, hence the claim follows by \cite[Theorem 14]{DS17}.
	\end{proof}
	
	We fix some assumptions on the Bayesian potential $\Phi$ to derive Lipschitz continuity of the map $\gd\to\bP_\gd$ with respect to the Hellinger distance.
	
	\begin{assumption}\label{ass:potential}
		~
		\begin{enumerate}[1)]
			\item\label{item:ass1}  For every $\gl>0$, there exists a constant $\gk_1(\gl)>0$ and a set $\gO_\gl\in\cA$ with $\bP(\gO_\gl)>0$ such that 
			\begin{align*}
				\Phi(\go; \gd) \le \gk_1(\gl),
				\quad \text{for all $\go\in\gO_\gl$ and $\gd\in B_\gl(0)$.}
			\end{align*}
			\item\label{item:ass2} For every $\gl>0$, there exists $\gk_2(\gl, \cdot)\in L^2(\gO, \bP)$ such that 
			\begin{align*}
				|\Phi(\go; \gd)-\Phi(\go; \gd')|\le \gk_2(\gl, \go)\|\gd-\gd'\|_2,
				\quad \text{for all $\go\in\gO$ and $\gd, \gd'\in B_\gl(0)$.}
			\end{align*}
		\end{enumerate}
	\end{assumption}

	\begin{prop}{\cite[Theorem 2.4]{hoang2012bayesian}}\label{prop:lipschitz-hell}
		Under Assumption~\ref{ass:potential}, there exists for any $\gl>0$ a constant $C(\gl)>0$ such that 
		\begin{equation}
			d_{\rm Hell}(\bP_\gd, \bP_{\gd'}) \le C(\gl)\|\gd-\gd'\|_2, \quad 
			\text{for all $\gd, \gd'\in B_\gl(0)$.}
		\end{equation}
	\end{prop}
	
	The proof of \cite[Theorem 2.4]{hoang2012bayesian} also yields a lower bound on the normalizing constants $Z(\gd)$ in~\eqref{eq:norm-constant}, that only depends on the norm of the data $\gd$: 
	
	\begin{cor}\label{cor:lower-bound-Z}
		Under Assumption~\ref{ass:potential}, there exists for any $\gl>0$ a constant $c(\gl)>0$ such that 
		\begin{equation}
			Z(\gd)\ge c(\gl)>0\quad
			\text{for all $\gd \in B_\gl(0)$.}
		\end{equation}
	\end{cor}
	
	\begin{proof}
		For fixed $\gl>0$ and any $\gd\in B_\gl(0)$ we have by Item~\eqref{item:ass1} of Assumption~\ref{ass:potential} 
		\begin{equation*}
			Z(\gd)=\int_\gO\exp\left(-\Phi(\go; \gd)\right)d\bP(\go)
			\ge \int_{\gO_\gl} d\bP(\go) \exp\left(-\gk_1(\gl)\right)
			=:c(\gl)
			>0.
		\end{equation*}
	\end{proof}
	
	We consider \emph{Besov random tree priors} as in \cite{KLSS21, SS22} in this article. 
	This particular prior has been used in \cite{SS22} to model the log-diffusion coefficient $b:=\log(a)$ in the elliptic forward problem~\eqref{eq:ellipticpde}. We review the construction of Besov random tree priors in the next section, and collect some results on well-posedness of~\eqref{eq:ellipticpdeweak} and regularity of pathwise weak solutions.
	
	\section{Besov Random Tree Priors}
	\label{sec:besov-rv}
	We introduce in this section Besov random tree priors and the associated elliptic forward problem to find $u$ in~\eqref{eq:ellipticpde} with $\log(a)$ given by a Besov random tree prior.
	We start by recalling some tools from multiresolution analysis (MRA) and the wavelet representation of Besov spaces. Thereafter we construct the Besov random tree prior, and record several results on well-posedness and regularity of the associated elliptic forward problem. We then discuss pathwise approximations by dimension truncation of the prior and finite elements in the last part of this section. The latter is in turn necessary to sample (approximately) from the posterior measure $\bP_\gd$ in Section~\ref{sec:BIP}.
	
	\subsection{MRA and wavelet representation of Besov spaces}
	\label{subsec:besovspaces}
	Let $\bT^d:=[0,1]^d$ denote the $d$-dimensional torus for $d\in\bN$.
	We briefly recall the construction of orthonormal wavelet basis on $L^2(\bR^d)$ and $L^2(\bT^d)$ 
	and the wavelet representation of the associated Besov spaces. 
	For more detailed accounts 
	we refer to \cite[Chapter 1]{TriebelFctSpcDom}, \cite[Chapter 1.2]{TriebelTOFS4}, 
	and to \cite[Chapter 5]{daubechies1992ten} for orthonormal wavelets in MRA.
	
	Let $\phi$ and $\psi$ be compactly supported scaling and wavelet functions in $\rC^{\ga}(\bR)$, $\ga\ge 1$,
	that are suitable for multi-resolution analysis in $L^2(\bR)$.  
	Further, we assume that $\psi$ satisfies \emph{the vanishing moment condition} 
	\begin{equation}\label{eq:vanmoments}
		\int_\bR \psi(x) x^m dx = 0, \quad m\in\bN_0,\; m<\ga.
	\end{equation}
	One example are Daubechies wavelets with $ M := \lfloor \ga \rfloor \in\bN$ vanishing moments 
	(also known as ${\mathrm D}{\mathrm B}(\lfloor \ga \rfloor)$-wavelets), 
	that have support $[-M+1,M]$ and are in $\rC^1(\bR)$ for $M\ge 5$ 
	(see, e.g., \cite[Section 7.1]{daubechies1992ten}).
	For any $j\in\bN_0$ and $k\in\bZ$, define the scaled and translated functions
	\begin{equation}
		\psi_{j,k,0}(x):=\phi(2^jx-k),
		\quad \text{and} \quad
		\psi_{j,k,1}(x):=\psi(2^jx-k),
		\quad x\in\bR.
	\end{equation}
	As $\|\phi\|_{L^2(\bR)}=\|\psi\|_{L^2(\bR)}=1$, it follows that 
	$((\psi_{0,k,0}), k\in\bZ)\,\cup\,((2^{j/2}\psi_{j,k,1}), (j,k)\in\bN_0\times\bZ)$ 
	is an orthonormal basis of $L^2(\bR)$.

	A corresponding \emph{isotropic}\footnote{Anisotropic tensorizations leading upon truncation 
		to so-called ``hyperbolic cross approximations'' may be considered. 
		As such constructions tend to
		inject preferred directions along the cartesian axes into approximations, 
		we do not consider them here.}
	wavelet basis that is orthormal in $L^2(\bR^d)$, $d\geq 2$ 
	may be constructed by tensorization of univariate MRAs.
	We define index sets $\cL_0:=\{0,1\}^d$ and $\cL_j:=\cL_0\setminus \{(0,\dots,0)\}$ for $j\in\bN$.
	We note that $\cL_j$ has cardinality $|\cL_j|=2^d$ if $j=0$, and $|\cL_j|=2^d-1$ otherwise.
	For any $l\in\cL_0$, we define furthermore  
	\begin{equation}\label{eq:psi_scaled}
		\psi_{j,k,l}(x):=2^{dj/2}\prod_{i=1}^d\psi_{j,k_i,l(i)}(x_i),
		\quad  j\in \bN_0,\; k\in\bZ^d,\; x\in\bR^d,
	\end{equation}
	to obtain that $((\psi_{j,k,l}),\; j\in \bN_0,\, k\in\bZ^d,\, l\in\cL_j)$ 
	is an orthonormal basis of $L^2(\bR^d)$.
	
	Orthonormal bases consisting of locally supported, 
	periodic functions on the torus $\bT^d$ can be introduced by
	tensorization, as e.g. in \cite[Section 1.3]{TriebelFctSpcDom}. 
	We utilize the construction in \cite[Section 2.1]{SS22}:
	Given $\phi$ and $\psi$, 
	we fix a scaling factor $w\in\bN$ such that 
	\begin{equation*}
		\text{supp}(\psi_{w,0,l})\subset 
		\left\{x\in\bR^d\big|\, \|x\|_2<\frac{1}{2}\right\},
		\quad l\in\cL_0.
	\end{equation*}
	With this choice of $w$, it follows for $j\in\bN_0$ that 
	\begin{equation*}
		\text{supp}(\psi_{j+w,0,l})\subset 
		\left\{x\in\bR^d\big|\, \|x\|_2<2^{-j-1}\right\}.
	\end{equation*}
	Now let 
	$K_{j}:=\{k\in\bZ^d|\,0\le k_1,\dots,k_d<2^{j}\}\subset 2^{j}\bT^d$ and note that $|K_{j+w}|=2^{d(j+w)}$.
	Define the one-periodic wavelet functions
	\begin{equation*}
		\psi_{j,k,l}^{per}(x):=\sum_{n\in\bZ^d} \psi_{j,k,l}(x-n), 
		\quad j\in\bN_0,\; k\in K_j,\;l\in\cL_0,\;x\in\bR^d,
	\end{equation*}
	and their restrictions to $\bT^d$ by 
	\begin{equation}\label{eq:torusbasis}
		\psi_{j,k}^{l}(x)
		:=
		\psi_{j,k,l}^{per}(x), \quad j\in\bN_0,\; k\in K_j,\;l\in\cL_0,\;x\in\bT^d.
	\end{equation}
	We now obtain for the index set $\cI_{w}:=\{j\in\bN_0,\; k\in K_{j+w},\;l\in\cL_j\}$ that 
	\begin{equation}
		\mathbf\Psi_w:=\left((\psi_{j+w,k}^l),\;(j,k,l)\in \cI_w \right)
	\end{equation}
	is a $L^2(\bT^d)$-orthonormal basis, see \cite[Proposition 1.34]{TriebelFctSpcDom}. 
	We further define the subspace 
	$V_{w+1}:=\text{span}\{\psi_{w,k}^l|\; k\in K_{w},\ l\in\cL_0\}\subset L^2(\bT^d)$ 
	and observe that $\dim(V_{w+1})=2^{d(w+1)}$.
	By the multiresolution analysis for one-periodic, univariate functions 
	in \cite[Chapter 9.3]{daubechies1992ten}, 
	it follows that 
	$( (\psi_{j,k}^l) ,\;j\le w,\; k\in K_j,\ l\in\cL_j)$ is another orthonormal basis of $V_{w+1}$.
	Hence, 
	we may replace the first $2^{d(w+1)}$ basis functions in~\eqref{eq:torusbasis} 
	to obtain the (computationally more convenient) 
	$L^2(\bT^d)$-orthonormal basis
	\begin{equation}\label{eq:torusbasis2}
		\mathbf\Psi:=\left((\psi_{j,k}^l),\;(j,k,l)\in \cI_{\mathbf \Psi} \right),
		\quad
		\cI_{\mathbf \Psi}:=\{j\in\bN_0,\; k\in K_j,\;l\in\cL_j\}.
	\end{equation}
	
	\begin{defi}\label{def:besovspace}
		Let $s>0$, $p\in[1,\infty]$ and $\varphi\in L^2(\bT^d)$.
		We define the \emph{Besov norms} 
		\begin{equation}\label{eq:besovnorm}
			\|\varphi\|_{B_{p,p}^s(\bT^d)}:=
			\left(
			\sum_{(j,k,l)\in\cI_{\mathbf\Psi}} 2^{jp(s+\frac{d}{2}-\frac{d}{p})} |(\varphi 	,\psi_{j,k}^l)_{L^2(\bT^d)}|^p
			\right)^{1/p},
			\quad
			p\in[1,\infty),
		\end{equation}
		and
		\begin{equation}\label{eq:besovnorminf}
			\|\varphi\|_{B_{\infty,\infty}^s(\bT^d)}:=
			\sup_{(j,k,l)\in\cI_{\mathbf\Psi}} 2^{j(s+\frac{d}{2})} |(\varphi ,\psi_{j,k}^l)_{L^2(\bT^d)}|
			<\infty.
		\end{equation}
		The corresponding \emph{Besov spaces} on $\bT^d$ are given by 
		\begin{equation}\label{eq:besovspace}
			B_{p,p}^s(\bT^d):=\{\varphi\in L^2(\bT^d)|\, \|\varphi\|_{B_{p,p}^s(\bT^d)}<\infty\}.
		\end{equation}
	\end{defi}

	We fix some notation for Besov, H\"older and Zygmund spaces to be used in the
	remainder of this paper.
	As the (periodic) domain $\bT^d$ does not vary in the subsequent analysis, 
	we use the abbreviations 
	$B_p^s:=B_{p,p}^s(\bT^d)$, $\rC^\ga:=\rC^\ga(\bT^d)$ and $\cC^\ga:=\cC^\ga(\bT^d)$
	for convenience in the following.

	\subsection{Besov random tree priors}
	\label{subsec:besov-rv}
	We introduce Besov random tree priors as wavelet expansions with respect to $\mathbf\Psi$, 
    where the $L^2(\bT^d)$-orthogonal projection coefficients are replaced by $p$-exponential random variables as a first step.
    To this end, let $p\in[1,\infty)$ and consider an independent and identically distributed (i.i.d.) sequence 
    $X=((X_{j,k}^l), (j,k,l)\in\mathbf{\cI_\Psi})$ of \emph{$p$-exponential} random variables. 
    That is, each $X_{j,k}^l$ is distributed with density 
    \begin{equation}\label{eq:p-exponential}
    	\phi_p(x):=\frac{1}{c_p}\exp\left(-\frac{|x|^p}{\gk}\right), \quad x\in\bR,
    	\qquad c_p := \int_\bR \exp\left(-\frac{|x|^p}{\gk}\right) dx,
    \end{equation}
    where $\gk>0$ is a fixed scaling parameter. 
    Let $\bQ_0$ denote the associated one-dimensional $p$-exponential measure on $(\bR, \cB(\bR))$.
    We recover the normal distribution with variance $\frac{\gk}{2}$ if $p=2$, 
    and the Laplace distribution with scaling $\gk$ for $p=1$.    
        
   	The random tree structure in our prior construction 
	is based on certain set-valued random variables, so-called \emph{Galton-Watson trees}. 
	Definitions of discrete trees, Galton-Watson (GW) trees, along with their basic properties,
	are given in Appendix A of \cite{SS22}, that treats the elliptic forward problem.
	
	\begin{defi}\cite[Definition 3]{KLSS21}\label{def:randomtree-prior}
		Let $s>0$, $p\in[1,\infty)$, and $X=((X_{j,k}^l), (j,k,l)\in\mathbf{\cI_\Psi})$ 
        be a sequence of $p$-exponentially distributed random variables.
		Let $\mfT$ denote the set of all trees with no infinite node (cf. \cite[Definition~A.1]{SS22},
		and let $T:\gO\to \mfT$ be a GW tree (cf. \cite[Definition~A.3]{SS22}
		with offspring distribution $\cP=\textrm{Bin}(2^d, \gb)$ for $\gb\in[0,1]$, and independent of $X$.
		Furthermore, let $\mfI_T$ be the set of wavelet indices associated to $T$ from~\cite[Equation~(79)]{SS22}.

		Define the \emph{random tree index set} 
		$\cI_T(\go):=\{(j,k,l)|\; (j,k)\in \mfI_T(\go),\; l\in\cL_j\}$ 
		and
		\begin{equation}\label{eq:randomtreeprior}
			b_T(\go):= \sum_{(j,k,l)\in \cI_T(\go)} \eta_j X_{j,k}^l(\go)\psi_{j,k}^l, \quad \go\in\gO,
			\quad\text{where}\quad
			\eta_j := 2^{-j(s+\frac{d}{2}-\frac{d}{p})}, 
			\quad j\in\bN_0.
		\end{equation}
		We refer to $b_T$ as a \emph{$B_p^s$-random variable with wavelet density $\gb$}.
	\end{defi}
	
	\begin{rem}\label{rem:Besov-random-tree}
		We obtain immediately the classical \emph{Besov priors} 
		as introduced in \cite{LassasBesov09} as special case with $\gb=1$,
		where $\cI_T(\go) = \cI_{\mathbf\Psi}$ holds almost surely.
		The series~\eqref{eq:randomtreeprior} 
		has a natural interpretation as orthogonal expansion of a random function 
		with respect to the (deterministic, fixed) basis $\mathbf \Psi$.
		The tree structure of $b_T$ gives rise to random fractals on $\bT^d$, 
		that occur whenever 
		the tree $T$ in Definition~\ref{def:randomtree-prior} does not terminate after a finite number of nodes. 
		It follows by \cite[Lemma~A.4]{SS22},
		that the latter event occurs with positive probability if $\gb\in(2^{-d},1]$. 
		In this case the Hausdorff dimension of the fractals is $d+\log_2(\gb)\in(0,d]$, 
		see \cite[Section 3]{KLSS21} for details.
	\end{rem}
	
	To treat elliptic inverse problems with $b_T$ as prior model, 
	we describe the corresponding probability space of parameters.
	Let $\bQ_0$ denote the univariate, $p$-exponential measure on $(\bR, \cB(\bR))$ 
	of the random variables $X_{j,k}^l$ with Lebesgue density as in~\eqref{eq:p-exponential}.
	The product-probability space of the $p$-exponentials $X$ is given by $(\gO_p, \cA_p, \bQ_p)$, 
	where 
	\begin{equation}
		\gO_p:=\bR^\bN,\quad  
		\cA_p:=\bigotimes_{n\in\bN} \cB(\bR),\quad \text{and} \quad
		\bQ_p:=\bigotimes_{n\in\bN} \bQ_0.
	\end{equation}
	Now let $s>0$ and $p\in[1,\infty)$ be fixed such that $s>\frac{d}{p}$. 
	We define the weighted $\ell^p$-spaces 
	\begin{equation*}
		\ell_s^p:=\left\{ x=\left(x_{j,k}^l, (j,k,l)\in\mathbf{\cI_\Psi}\right)\in\bR^\bN|\;
		\|x\|_{s,p}<\infty\right\},
	\end{equation*}
	where 
	\begin{equation*}
		\|x\|_{s,p}
		:= \left(\sum_{(j,k,l)\in\mathbf{\cI_\Psi}} 2^{-jps}|x_{j,k}^l|^p \right)^{1/p}.
	\end{equation*}
	As $1\le p<\infty$, $(\ell_s^p, \left\|\cdot\right\|_{s,p})$ is a separable Banach space.
    We observe that for $X\sim \bQ_p$ it holds
	\begin{align*}
		\bE(\|X\|_{s,p}^p)
		\le
		\sum_{(j,k,l)\in\mathbf{\cI_\Psi}} 2^{-jps}\bE(|X_{j,k}^l|^p) 
		\le 
		C \sum_{j=0}^\infty 2^{-jps} 2^{dj} (2^d-1) 
		\le 
		C \sum_{j=0}^\infty 2^{-jp(s-\frac{d}{p})}
		<\infty, 
	\end{align*}
	since $s>\frac{d}{p}$, thus $\bQ_p$ is concentrated on $\ell_s^p$. 
	Therefore, we may regard $(\ell_s^p, \cB(\ell_s^p), \bQ_p)$ as 
        probability space of random coefficient sequences $X$ in the expansion~\eqref{eq:randomtreeprior}. 
	
	The set-valued random variable $T$ is a GW tree, and hence takes values in the Polish space $(\mfT, \gd_{\mfT})$ 
	of all trees with no infinite node. 
	The metric $\gd_\mfT$ and the associated Borel $\gs$-algebra $\cB(\mfT)$ with respect to $\mfT$ 
	can be expressed explicitly \cite[Def.~A2]{SS22}, or in \cite[Sec.~2.1]{abraham2015GW-Trees}.
	The image measure $\bQ_T$ of the GW tree $T$ on $(\mfT, \cB(\mfT))$ then solely depends on the parameters $\gb$ and $d$ of the offspring distribution $\cP=\textrm{Bin}(2^d, \gb)$, 
	and is given in \cite[Equation~(77)]{SS22}.
	Hence, the parameter probability space of GW trees is given by $(\mfT, \cB(\mfT), \bQ_T)$.
	
	To combine the random coefficients $X$ with the GW tree $T$, we define the cartesian product 
	$\gO:=\ell_s^p\times \mfT$ and equip $\gO$ with the metric 
	\begin{equation*}
		d_\gO((x_1, {\bf t_1}),(x_2, {\bf t_2})) := \|x_1-x_2\|_{s,p}+\gd_\mfT({\bf t_1}, {\bf t_2}). 
	\end{equation*}
	
	\begin{prop}
		The space $(\gO, d_\gO)$ is Polish with Borel $\gs$-algebra 
		given by $\cB(\gO)=\cB(\ell_s^p\times \mfT)=\cB(\ell_s^p)\otimes\cB(\mfT)$.
	\end{prop}
	\begin{proof}
		By~\cite[Lemma~2.1]{abraham2015GW-Trees} the metric space $(\mfT, \gd_\mfT)$ 
		with $\gd_\mfT$ given in \cite[Def.~A.2]{SS22} is complete and separable. 
		Separability and completeness of $(\gO, d_\gO)$ follows then by \cite[Corollary 3.39]{AB06}. 
		Furthermore,  
		$\cB(\gO)=\cB(\ell_s^p\times\mfT)=\cB(\ell_s^p)\otimes\cB(\mfT)$ 
		holds by \cite[Theorem 4.44]{AB06}.
	\end{proof}
	
	We are now ready to define the prior probability space associated to the 
	$\ell_s^p\times \mfT$-valued random variable $(X, T)$: 
	Let $(\gO, \cA, \bP)$ denote the product probability space given by 
	\begin{equation}
		\gO:=\ell_s^p\times \mfT,\quad  
		\cA:=\cB(\ell_s^p)\otimes\cB(\mfT),\quad \text{and} \quad
		\bP:=\bQ_p \otimes \bQ_T.
	\end{equation}
	We remark that the product structure of the measure $\bP=\bQ_p \otimes \bQ_T$ 
	is tantamount to stochastic independence of $X$ and $T$.
	
	It still remains to identify a realization of the random variable $(X,T)$ 
	with the corresponding random tree prior $b_T$. 
	To this end, we consider the canonical mapping
	\begin{equation}
		b_T:\gO \to L^2(\bT^d),\quad 
		\go\mapsto \sum_{(j,k,l)\in \cI_T(\go)} \eta_j X_{j,k}^l(\go)\psi_{j,k}^l.
	\end{equation}
	The map $b_T:\gO \to L^2(\bT^d)$ is indeed well-defined 
	since $\|b_T\|_{L^2(\bT^d)}<\infty$ holds due to $s>\frac{d}{p}$.
	Moreover, $b_T$ is $\cA/\cB(L^2(\bT^d))$-measurable, 
	as is seen in Proposition~\ref{prob:measurability} below. 
	Therefore, 
	the pushforward probability measure of $b_T$ under the prior measure $\bP$ is given via
	\begin{equation}\label{eq:beson-measure}
		b_T\#\bP(B):=\bP(b_T^{-1}(B)),\quad B\in \cB(L^2(\bT^d)).
	\end{equation}
	The associated probability space of $B_p^s$-random variables $b_T$ with wavelet density $\gb$ is given by 
	\begin{equation*}
		(L^2(\bT^d),\,\cB(L^2(\bT^d)),\,b_T\#\bP).
	\end{equation*}
	We know from~\cite[Remark 2.9]{SS22} that $b_T\#\bP$ is concentrated on $B_p^t$ for any $t\in(0, s-\frac{d}{p})$.
	A more refined result that concentrates $b_T\#\bP$ on Besov spaces $B_q^t$ for $q\ge 1$ 
	with smoothness index $t=t(s,d,p,\gb,q)$ is given in Theorem~\ref{thm:random-tree-reg} below.
	We conclude this section by two results on measurability and pathwise regularity of $b_T$.
	
	\begin{prop}{\cite[Proposition 2.10]{SS22}}\label{prob:measurability}
		Let $s>\frac{d}{p}$, $\gb\in[0,1]$, and 
		let $b_T$ be a $B_p^s$-random variable with wavelet density $\gb$.
		Then $b_T:\gO\to \rC(\bT^d)$ and $b_T$ is (strongly) $\cA/\cB(\rC(\bT^d))$-measurable.
	\end{prop}
	
	\begin{thm}{\cite[Theorem 2.11]{SS22}}\label{thm:random-tree-reg}
		Let $b_T$ be a $B_p^s$-random variable 
		with wavelet density $\gb=2^{\gg-d}$ as in Definition~\ref{def:randomtree-prior} with $\gg\in(-\infty, d]$.
		\begin{enumerate}[1.)]
			\item It holds that $b_T\in L^q(\gO, \bP; {B^t_q})$, and hence $b_T\in B_q^t$ $P$-a.s., 
			for all $t>0$ and $q\ge1$ such that $t < s+ \frac{d-\gg}{q} - \frac{d}{p}$.
			\item Let $s-\frac{d}{p}>0$ and $t\in(0,s-\frac{d}{p})$. Then there is a $\eps_p>0$ such that 
			\begin{equation*}
				\bE_\bP\left(\exp\left(\eps \|b\|_{\cC^t}^p\right)\right) < \infty,\quad \eps\in(0,\eps_p),
			\end{equation*}
			In particular, it holds $b_T\in L^q(\gO, \bP; \cC^t)$ for any $q\ge1$.
			\item Let $q\ge 1$ and $s-\frac{d}{p}-\frac{\min(\gg,0)}{q}>0$. 
			For any $t\in(0,s-\frac{d}{p}-\frac{\min(\gg,0)}{q})$ it holds $b_T\in L^q(\gO, \bP; \cC^t)$. 
		\end{enumerate}   
	\end{thm}

	\subsection{Well-posedness and regularity of forward problem}
	
	Let $\cD\subset\bR^d$, $d\in\{1,2,3\}$ be a convex polygonal domain, 
	with the boundary $\partial \cD$ consisting of a finite number of line or plane segments.
	We assume furthermore that $\cD\subseteq \bT^d$. 
	Let $\varphi|_{\cD}$ denote the \emph{restriction} of any $\varphi\in\rS'(\bR^d)$ to $\cD$, which is in turn given by the element $\varphi|_{\cD}\in \rD'(\cD)$ such that 
	\begin{equation*}
		\dualpair{\rD'(\cD)}{\mathrm D(\cD)}{\varphi|_{\cD}}{v}
		=
		\dualpair{\rS'(\bR^d)}{\rS(\bR^d)}{\varphi}{v_0},
		\quad 
		v\in\mathrm D(\cD),
	\end{equation*}
	where $v_0\in\mathrm D(\bR^d)\subset \rS(\bR^d)$ denotes the zero-extension of any $v\in \mathrm D(\cD)$ (cf. \cite[Section 2]{TriebelFctSpcDom}). According to \cite[Theorem 1.29]{TriebelFctSpcDom} there exists for any $b\in B_p^s$ a unique, one-periodic extension ${\rm ext}^{per}(b):\bR^d\to \bR$, so that $b = {\rm ext}^{per}(b)|_\cD$, see also \cite[Section 4.2]{SS22} for further details.
	
	The restriction of $b_T$ given in Definition~\ref{def:randomtree-prior} to $\cD\subseteq\bT^d$ is thus given by
	\begin{equation}\label{eq:restrictedbesovprior}
		b_{T,\cD}(\go)
		:=
		(\mathrm{ext}^{per}b_T(\go))|_{\cD}.
	\end{equation}
	We call $b_{T,\cD}$ a \textit{$B_p^s(\cD)$-valued random variable}.
	\begin{rem}
		\label{rmk:Domain}
		Note that $b_{T,\cD}$ is not (necessarily) periodic if $\cD\subsetneq\bT^d$, but \emph{merely the restriction of a periodic function} from the torus $\bT^d$. 
		Assuming $\cD\subseteq\bT^d$ for the sake of brevity does not have any 
		substantial impact on the following results: In case that $\cD\not\subset\bT^d$ is a bounded domain, we could extend Definition~\eqref{def:randomtree-prior} from the torus $\bT^d$ to a sufficiently large (periodic) domain $[-L,L]^d$, with $L>1$ such that $\cD\subset[-L,L]^d$. We would then simply define $b_{T,\cD}$ as the restriction of a $L$-periodic function on this enlarged domain.   
	\end{rem}
	
	Now we set $a=\exp(b_{T,\cD})$ in~\eqref{eq:ellipticpde} to obtain the \emph{elliptic forward problem with Besov random tree prior} to find $u(\go):\cD\to \bR$ for given $\go\in\gO$ such that 
	\begin{equation}\label{eq:elliptic-pde-besov}
		\begin{alignedat}{2}
			-\nabla\cdot(\exp(b_{T,\cD}(\go))\nabla u(\go)) &= f\quad &&\text{in $\cD$}, \quad 
			u(\go) = 0 \quad &&\text{on $\partial\cD$}.
		\end{alignedat}
	\end{equation}
	
	\begin{thm}\cite[Theorem 3.9]{SS22}\label{thm:solution-regularity}
		Let $b_{T,\cD}$ be given in~\eqref{eq:restrictedbesovprior} for $p\in[1,\infty)$, $s>0$ and $\gb\in[0,1]$, so that $sp>d$.
		Furthermore, let $f\in V'$. Then the following assertions hold.
		
		\begin{enumerate}[1.)]
			\item 	 There exists almost surely a unique weak solution 
                                 $u(\go)\in V$ to~\eqref{eq:elliptic-pde-besov} and 
                                 $u:\gO\to V$ is strongly measurable.
			
			\item For sufficiently small $\gk>0$ in~\eqref{eq:p-exponential}, 
			there are constants $\ol q\in(1,\infty)$ and $C>0$ such that
			\begin{equation*}
				\|u\|_{L^q(\gO, \bP; V)}\le C \|f\|_{V'}  <\infty
				\quad
				\begin{cases}
					&\text{for $q\in[1,\ol q)$ if $p=1$, and} \\
					&\text{for any $q\in [1,\infty)$ if $p>1$}.
				\end{cases}
			\end{equation*}
			
			\item 
			Let $f\in H$, $r\in (0,s-\frac{d}{p})\cap(0,1]$, let $r_0\in(0,r)$ if $r<1$, and $r_0=1$ if $r=1$.
			For sufficiently small $\gk>0$ in~\eqref{eq:p-exponential},  
			there are constants $\ol q\in(1,\infty)$ and $C>0$ such that .	
			\begin{equation*}
				\|u\|_{L^q(\gO, \bP; H^{1+r_0}(\cD))} \le C \|f\|_{H} <\infty
				\quad
				\begin{cases}
					&\text{for $q\in[1,\ol q)$ if $p=1$ and} \\
					&\text{for any $q\in [1,\infty)$ if $p>1$}.
				\end{cases}
			\end{equation*}
		\end{enumerate}
		
	\end{thm} 
    We observe that the non-negative diffusion coefficient in the forward PDE
    \eqref{eq:elliptic-pde-besov} is \emph{not lower-bounded away from zero.}
	We also remark that the condition "for sufficiently small $\gk>0$" 
    in part 2.) and 3.) of Theorem~\ref{thm:solution-regularity} 
    only applies for $p=1$, 
	and ensures that $\exp(\|b_{T,\cD}\|_{L^\infty(\cD)})\in L^q(\gO)$, 
    or, respectively, $\exp(\|b_{T,\cD}\|_{\rC^r(\ol\cD)})\in L^q(\gO)$, 
    for some ($\kappa$-dependent) $q\ge 1$. 
	
	\subsection{Pathwise approximation of the forward problem}
    \label{sec:PthwAppr}

	To obtain a tractable approximation of $b_T$ in~\eqref{eq:randomtreeprior}, 
        we truncate the wavelet series expansion after $N\in\bN$ scales to obtain the $\emph{truncated random tree Besov prior}$ 
	\begin{equation}\label{eq:randomtreepriortrunc}
		b_{T,N}(\go):= \sum_{\substack{(j,k,l)\in \cI_T(\go) \\ j\le N}} \eta_jX_{j,k}^l(\go)\psi_{j,k}^l, \quad \go\in\gO.
	\end{equation}
	The corresponding diffusion problem in weak form with truncated coefficient 
    for fixed $\go\in\gO$ is to find  $u_N(\go)\in V$ such that for all $v\in V$
	\begin{equation}\label{eq:ellipticpdetrunc}
		\int_\cD a_N(\go)\nabla u_N(\go)\cdot\nabla v dx = \dualpair{V'}{V}{f}{v},
	\end{equation}
	where 
	\begin{equation}\label{eq:diffusioncoefftrunc}
		a_N:\gO\to L^\infty(\cD), \quad \go\mapsto \exp(b_{T,N}(\go)|_{\cD}).
	\end{equation}
	The solution $u_N:\gO\to V$ to Problem~\eqref{eq:ellipticpdetrunc} 
	with truncated coefficient is still not fully tractable, 
	as it takes values in the infinite-dimensional Hilbert space $V$. 
	Thus, we consider Galerkin-finite element approximations of $u_N$ for a fixed truncation index $N$ in the remainder of this section. 
	
	As a first step, we discretize the convex domain 
	$\cD\subset \bT^{d}$, $d\in\{1,2,3\}$ by a 
	sequence of simplices (intervals/triangles/tetrahedra) 
	or parallelotopes (intervals/ parallelograms/parallelepipeds), 
	denoted by $(\cK_h)_{h\in \mfH}$. 
	The refinement parameter $h>0$ takes values in a countable index set $\mfH \subset (0,\infty)$
	and corresponds to the longest edge of a simplex/parallelotope $K\in\cK_h$. 
	We impose the following assumptions on $(\cK_h)_{h\in \mfH}$ 
	to obtain a sequence of "well-behaved" triangulations.
	\begin{assumption}\label{ass:triangulation}
		The sequence $(\cK_h)_{h\in \mfH}$ satisfies:
		\begin{enumerate}
			\item \emph{Admissibility:} 
			For each $h\in \mfH$, $\cK_h$ consists of open, 
			non-empty simplices/parallelotopes $K$ 
			such that 
			\begin{itemize}
				\item $\ol \cD = \bigcup_{K\in\cK_h} \ol K$, 
				\item $K_1\cap K_2=\emptyset$ for any two $K_1, K_2\in \cK_h$ such that $K_1\neq K_2$, and 
				\item the intersection $\ol K_1\cap \ol K_2$ for $K_1\neq K_2$ is 
				either empty, a common edge, a common vertex, 
				or (in space dimension $d=3$) a common face of $K_1$ and $K_2$.
			\end{itemize}
			\item \emph{Shape-regularity:} 
			Let $\rho_{K,in}$ and $\rho_{K,out}$ denote the radius of the inner and outer circle, respectively, for a given $K\in\cK_h$.  
			There is a constant $\rho > 0$ such that 
			\begin{equation*}
				\rho : = \sup_{h\in \mfH} \sup_{K\in\cK_h} \frac{\rho_{K,out}}{\rho_{K,in}} < \infty.
			\end{equation*}
		\end{enumerate}
	\end{assumption} 
	
	Based on a given tesselation $\cK_h$, we define the space of piecewise (multi-)linear finite elements 
	\begin{equation*}
		V_h:=
		\begin{cases}
			\{v\in V|\; \text{$v|_T$ is linear for all $K\in\cK_h$}\},
			&\quad\text{if $\cK_h$ consists of simplices}, \\
			\{v\in V|\; \text{$v|_T$ is $d$-linear for all $K\in\cK_h$}\},
			&\quad\text{if $\cK_h$ consists of parallelotopes}.
		\end{cases}
	\end{equation*}
	Clearly, $V_h\subset V$ is a finite-dimensional space and we define $n_h:=\dim(V_h)\in \bN$.
	This yields for fixed $\go\in\gO$ the \emph{fully discrete problem} to find $u_{N,h}(\go)\in V_h$ such that for all $v_h\in V_h$
	\begin{equation}\label{eq:ellipticpdediscrete}
		\int_\cD a_N(\go)\nabla u_{N,h}(\go)\cdot\nabla v_h dx = \dualpair{V'}{V}{f}{v_h}.
	\end{equation}
	The combined truncation and FE-approximation error is bounded by the next result.

	\begin{thm}\cite[Theorems 4.4., 4.7 and 4.8]{SS22}\label{thm:forward-error}
		Let $(\cK_h)_{h\in \mfH}$ be a sequence of triangulations satisfying Assumption~\ref{ass:triangulation}, 
		and let $u$, $u_N$ and $u_{N,h}$ be the pathwise weak solutions to 
		\eqref{eq:elliptic-pde-besov},~\eqref{eq:randomtreepriortrunc} and \eqref{eq:ellipticpdediscrete} for given $N\in\bN$ and $h\in\mfH$.
		Furthermore, let $p\in[1,\infty)$ and $s>0$ such that $sp>d$.
		
		For any $f\in H$, sufficiently small $\gk>0$ in~\eqref{eq:p-exponential}, 
		any $r\in (0,s-\frac{d}{p})\cap (0,1]$ and $t\in(0,s-\frac{d}{p})$,
		there are constants $\ol q\in(1,\infty)$ and $C>0$ such that for any $N\in\bN$ and $h\in\mfH$ there holds
		\begin{align*}
			\|u-u_N\|_{L^q(\gO, \bP; V)}&\le C 2^{-Nt}
			\quad
			\begin{cases}
				&\text{for $q\in[1,\ol q)$ if $p=1$,} \\
				&\text{for any $q\in [1,\infty)$ if $p>1$,}
			\end{cases}
			\\
			\|u_N-u_{N,h}\|_{L^q(\gO, \bP; V)}&\le C h^r 
			\,\qquad
			\begin{cases}
				&\text{for $q\in[1,\ol q)$ if $p=1$,} \\
				&\text{for any $q\in [1,\infty)$ if $p>1$,}
			\end{cases}
			\\
			\|u_N-u_{N,h}\|_{L^q(\gO, \bP; H)}&\le C h^{2r} 
			\quad\;\;\;
			\begin{cases}
				&\text{for $q\in[1,\ol q)$ if $p=1$,} \\
				&\text{for any $q\in [1,\infty)$ if $p>1$.}
			\end{cases}
	\end{align*}		
	\end{thm}
	
	\section{Inverse Problem with Besov Random Tree Prior}
	\label{sec:BIP-random-tree}
	
	Consider again the Bayesian inverse problem setting from Section~\ref{sec:BIP}, 
	where we assume that $\log(a)$ in~\eqref{eq:ellipticpde} is given by a 
	Besov random tree prior, i.e. $a=\exp(b_{T,\cD})$. 
	We verify Assumption~\ref{ass:potential} in this setting to 
	ensure that Proposition~\ref{prop:lipschitz-hell} and Corollary~\ref{cor:lower-bound-Z} are valid. 
	
	\begin{lem}\label{lem:local-boundedness}
		Let the assumptions of Theorem~\ref{thm:solution-regularity} hold with $q\ge2$, let $\log(a)$ in~\eqref{eq:ellipticpde} be given by a Besov random tree prior as $a=\exp(b_{T,\cD})$, and let $\rho$ be given in~\eqref{eq:Gauss-noise}. Then, Assumption~\ref{ass:potential} holds for $\Phi$ and $\rho$.
	\end{lem}
	
	\begin{proof}
		By Theorem~\ref{thm:solution-regularity} we obtain that $u\in L^q(\gO, \bP; V)$. 
		This implies that there is a constant $C_u>0$ and a set $\gO_u\in\cA$ with $\bP(\gO_u)>0$, such that $\|u(\go)\|_V\le C_u$ for all $\go\in \gO_u$.
		Now let $\gl>0$ be fixed and define $\chi:=\gl+\|\cO\|_{((H_0^{\gt_\cO})')^k}C_u$.
		For all $\go\in \gO_u$ and $\gd\in B_\gl(0)$ we then obtain
		\begin{align*}
			\Phi(\go; \gd)
			\le
			|\log(\rho(\gd-[\cO\circ u](\go)))|
			\le 
			\sup_{x\in B_\chi(0)} |\log(\rho(x))|
			\le
			\frac{k}{2}|\log(2\pi\det(\gS))|
			+\frac{1}{2} \|\gS^{-1}\|_2\chi^2
			<\infty.
		\end{align*}
		Hence, $\Phi$ satisfies Item~\ref{item:ass1} of Assumption~\ref{ass:potential} with $\gO_\gl:=\gO_u$ and
		\begin{equation*}
			\gk_1(\gl):=\frac{k}{2}|\log(2\pi\det(\gS))|
			+\frac{1}{2} \|\gS^{-1}\|_2(\gl+\|\cO\|_{((H_0^{\gt_\cO})')^k}C_u)^2.
		\end{equation*}
		To show Item~\ref{item:ass2}, we fix $\go\in\gO$ and let $\gd, \gd'\in B_\gl(0)$.
		By~\eqref{eq:Gauss-noise}, we obtain
		\begin{align*}
			|\Phi(\go; \gd)-\Phi(\go; \gd')|
			&= 
			\left|
			\log\left(
			\frac{\rho(\gd-[\cO\circ u](\go))}{\rho(\gd'-[\cO\circ u](\go))}
			\right)
			\right|\\
			&=
			\frac{1}{2}\left|
			-\gd^\top\gS^{-1}\gd + (\gd')^\top\gS^{-1}(\gd')
			+2[\cO\circ u](\go)^\top\gS^{-1}(\gd-\gd')
			\right| \\
			&\le
			\frac{1}{2}\left|
			-\gd^\top\gS^{-1}(\gd-\gd') + (\gd')^\top\gS^{-1}(\gd'-\gd)\right|
			+
			\left|[\cO\circ u](\go)^\top\gS^{-1}(\gd'-\gd)
			\right| \\
			&\le \left(\max(\|\gd\|_2, \|\gd'\|_2)+\|\cO\|_{((H_0^{\gt_\cO})')^k}\|u(\go)\|_V\right)
			\|\gS^{-1}\|_2 \|\gd-\gd'\|_2 \\
			&\le \left(\gl+\|\cO\|_{((H_0^{\gt_\cO})')^k}\|u(\go)\|_V\right)
			\|\gS^{-1}\|_2 	\|\gd-\gd'\|_2.	
		\end{align*}
		Since $u\in L^2(\gO,\bP;V)$ by Theorem~\ref{thm:solution-regularity}, Item~\ref{item:ass2} of 
		Assumption~\ref{ass:potential} holds with 
		\begin{equation*}
			\gk_2(\gl,\go):=
			\left(\gl+\|\cO\|_{((H_0^{\gt_\cO})')^k}\|u(\go)\|_V\right)
			\|\gS^{-1}\|_2.
		\end{equation*}
	\end{proof}
	
	\begin{rem}
		Let $\gl_{\rm min}>0$ denote the smallest eigenvalue of $\gS$.
		From the proof of Lemma~\ref{lem:local-boundedness} it is apparent that $\gk_1(\gl)=\cO(\gl_{\rm min}^{-1})$, $\gk_2(\gl,\go)=\cO(\gl_{\rm min}^{-1})$ for fixed $\go$ and $\gl$.
		Hence, the bounds in 
		Assumption~\ref{ass:potential} deteriorate in the "small noise-limit" when $\gl_{\rm min}\to 0$. 
		Consequently, $C(\gl)=\cO(\gl_{\rm min}^{-1})$ for fixed $\gl>0$ in Proposition~\ref{prop:lipschitz-hell}, and $c(\gl)=\cO(\gl_{\rm min}^{\frac{k}{2}})$ in Corollary~\ref{cor:lower-bound-Z} with the choice of $\gk_1$ as in the proof of Lemma~\ref{lem:local-boundedness}.
		
	\end{rem}
	
	The (exact) posterior $\bP_\gd$ from~\eqref{eq:RN_derivative} is in general out of reach, 
	as only biased samples $u_{N,h}\approx u$ as in~\eqref{eq:ellipticpdediscrete} of the forward problem are available. 
	Therefore, we consider the \emph{approximated posterior}   
	\begin{equation}\label{eq:discrete_posterior}
		d\bP_{\gd, N, h}(\go) := \frac{\exp\left(-\Phi_{N,h}(\go; \gd)\right)d\bP(\go)}
		{Z_{N,h}(\gd)}, 
	\end{equation}
	with discrete Bayesian potential and normalizing constant given by 
	\begin{equation*}
		\Phi_{N,h}(\go; \gd)
		:=
		-\log\left(\rho(\gd-[\cO\circ u_{N,h}](\go))\right),
		\quad\text{and}\quad 
		Z_{N,h}(\gd):=\int_\gO\exp\left(-\Phi_{N,h}(\go; \gd)\right)d\bP(\go)>0.
	\end{equation*}
	
	\begin{prop}\label{prop:posterior_approx}
		Let the assumptions of Theorem~\ref{thm:forward-error} hold such that $\ol q\ge4$ in case that $p=1$. Then, for any $\gl>0$ there is a $C(\gl)>0$, independent of $N$ and $h$, such that 
		\begin{equation*}
			d_{\rm Hell}(\bP_\gd, \bP_{\gd, N, h})\le C(\gl)(2^{-Nt}+ h^{(2-\gt_\cO)r})
			\quad\text{for all $\gd\in B_\gl(0)$.}
		\end{equation*}
	\end{prop}

	We need a uniform lower bound on $Z_{N,h}(\gd)$ to prove Proposition~\ref{prop:posterior_approx}:
	\begin{lem}\label{lem:lower-bound-Z-disc}
		Let the assumptions of Theorem~\ref{thm:forward-error} hold. Then, for any $\gl>0$ there is a constant $c(\gl)>0$, independent of $N$ and $h$, such that 
		\begin{equation}
			Z_{N,h}(\gd)\ge c(\gl)>0\quad
			\text{for all $\gd \in B_\gl(0)$.}
		\end{equation}
	\end{lem}
	
	\begin{proof}
		By Theorem~\ref{thm:forward-error}, there is a constant $\widetilde C>0$ such that for all $N\in\bN$ and $h\in\mfH$ there holds 
		\begin{equation}\label{eq:L1error}
			\|u-u_{N,h}\|_{L^1(\gO, \bP; V)}\le C(2^{-Nt}+h^r) \le \widetilde C<\infty.
		\end{equation}
		Since $u\in L^1(\gO, \bP; V)$, this implies by the reverse triangle inequality that 
		\begin{equation}\label{eq:L1bound}
			\|u_{N,h}\|_{L^1(\gO, \bP; V)}\le \widetilde C+\|u\|_{L^1(\gO, \bP; V)}<\infty.
		\end{equation}
		Now define $\widetilde C_u:= 2(\widetilde C+\|u\|_{L^1(\gO, \bP; V)})>0$, and the set 
		\begin{equation*}
			\widetilde \gO_u:=\{\go\in\gO\,\big| \|u_{N,h}\|_{V}\le \widetilde C_u\}\,\in \cA.
		\end{equation*}
		The set $\widetilde \gO_u$ depends in general on $N$ and $h$. We have
		by Markov's inequality, Inequality~\eqref{eq:L1bound} and the definition of $\widetilde C_u$ that
		\begin{equation}\label{eq:Omega-bound}
			\bP(\widetilde \gO_u)
			=1- \bP(\|u_{N,h}\|_{V}> \widetilde C_u)
			\ge 1 - \frac{\|u_{N,h}\|_{L^1(\gO, \bP; V)}}{\widetilde C_u}
			\ge \frac{1}{2},
		\end{equation}
		holds for all $N\in\bN$ and $h\in\mfH$.
		Now let $\gl>0$ be fixed. 
		For all $\go\in \widetilde\gO_u$ and $\gd\in B_\gl(0)$ we obtain
		\begin{align*}
			\Phi_{N,h}(\go; \gd)
			\le
			|\log(\rho(\gd-[\cO\circ u_{N,h}](\go)))|
			\le 
			\sup_{x\in B_\chi(0)} |\log(\rho(x))|<\infty,
		\end{align*}
		where $\chi:=\gl+\|\cO\|_{((H_0^{\gt_\cO})')^k}\widetilde C_u$ and the last estimate is finite by continuity of $\log\circ\rho:\bR^k\to\bR$.
		Hence, $\Phi_{N,h}$ satisfies the first part of Assumption~\ref{ass:potential} with $\gO_\gl:=\widetilde\gO_u$ with 
		\begin{equation*}
			\gk_1(\gl):=\sup_{x\in B_\chi(0)} |\log(\rho(x))|<\infty,
		\end{equation*}
		due to~\eqref{eq:Omega-bound}.
		The claim now follows analogously to the proof of Corollary~\ref{cor:lower-bound-Z}, since $\bP(\widetilde \gO_u)$ and $\gk_1(\gl)$ are bounded uniformly in $N$ and $h$.
	\end{proof}
	
	\begin{proof}[Proof of Proposition~\ref{prop:posterior_approx}:]
		The proof basically follows the proof of \cite[Proposition 10]{hoang2013complexity}, where we substitute the estimate from Theorem~\ref{thm:forward-error} at the appropriate positions. 
		Since both $\bP_\gd$ and $\bP_{\gd, N, h}$ are absolutely continuous with respect to $\bP$, we have for fixed $\gl>0$ and any $\gd\in B_\gl(0)$ that
		\be\label{eq:exp-taylor1}
		\begin{split}
			&2d_{\rm Hell}(\bP_\gd, \bP_{\gd, N, h})^2 \\
			&= 
			\int_\gO \left(
			Z(\gd)^{-1/2}\exp\left(-\frac{1}{2}\Phi(\go; \gd)\right)
			-
			Z_{N,h}(\gd)^{-1/2}\exp\left(-\frac{1}{2}\Phi_{N,h}(\go; \gd)\right)
			\right)^2d\bP(\go)\\
			&\le 
			2\int_\gO Z(\gd)^{-1}
			\left(\exp\left(-\frac{1}{2}\Phi(\go; \gd)\right)
			-\exp\left(-\frac{1}{2}\Phi_{N,h}(\go; \gd)\right)
			\right)^2d\bP(\go) \\
			&\quad +
			2\int_\gO\left(Z(\gd)^{-1/2}-Z_{N,h}(\gd)^{-1/2}\right)^2
			\exp\left(-\Phi_{N,h}(\go; \gd)\right)\bP(\go)\\
			&:=2(I+II).
		\end{split}
		\ee
		To bound $I$ we use Taylor-expansion and that $\gS$ in~\eqref{eq:potential} is positive definite to obtain 
			\begin{align*}
				&\left|\exp\left(-\frac{1}{2}\Phi(\go; \gd)\right)
				-\exp\left(-\frac{1}{2}\Phi_{N,h}(\go; \gd)\right)\right|\\
				&\quad\le 
				\frac{1}{2}
				\exp\left(-\frac{1}{2}\frac{k\log\left(2\pi\det(\gS)\right)}{2}\right)
				\left|\Phi(\go; \gd)-\Phi_{N,h}(\go; \gd)\right|\\
				&\quad\le C \left|\Phi(\go; \gd)-\Phi_{N,h}(\go; \gd)\right|,
			\end{align*}
			where $C>0$ is a deterministic constant.
		Now let $C_\rho\in(0,\infty)$ be the Lipschitz constant of $\rho$ in~\eqref{eq:Gauss-noise}.
		We use the bound in~\eqref{eq:exp-taylor1} together with Hölder's inequality to bound the first term by
		\begin{align*}
			I&\le 
			C \int_\gO Z(\gd)^{-1}
			\left|\Phi(\go; \gd)-\Phi_{N,h}(\go; \gd) \right|^2 d\bP(\go)
			\\
			&=C
			\int_\gO Z(\gd)^{-1}
			\left|\rho(\gd - [\cO\circ u](\go))
			-\rho(\gd - [\cO\circ u_{N,h}](\go))\right|^2d\bP(\go) \\
			&\le 
			C\frac{C_\rho^2\|\cO\|_{((H_0^{\gt_\cO})')^k}^2}{Z(\gd)}
			\int_\gO \|u(\go)-u_{N,h}(\go)\|_{H^{\gt_\cO}}^2 d\bP(\go) \\
			&\le 
			\frac{C}{Z(\gd)}\int_\gO \|u(\go)-u_N(\go)\|_V^2
			+\|u_N(\go)-u_{N,h}(\go)\|_{H^{\gt_\cO}}^2 d\bP(\go) 
			\\
			&\le 
			\frac{C}{Z(\gd)}\int_\gO \|u(\go)-u_N(\go)\|_V^2
			+\|u_N(\go)-u_{N,h}(\go)\|_V^{2\gt_\cO} 
			\,\|u_N(\go)-u_{N,h}(\go)\|_H^{2(1-\gt_\cO)} d\bP(\go) 
			\\
			&\le 
			\frac{C}{Z(\gd)}\left(
			\|u-u_N\|_{L^2(\gO, \bP; V)}^2
			+\|u_N-u_{N,h}\|_{L^{4\gt_\cO}(\gO, \bP; V)}^{2\gt_\cO}
			\|u_N-u_{N,h}\|_{L^{4(1-\gt_\cO)}(\gO, \bP; H)}^{2(1-\gt_\cO)}
			\right).
		\end{align*}	
		We then use Theorem~\ref{thm:forward-error} with $q:=4\max(\gt_\cO,(1-\gt_\cO))\le 4$ to derive the estimate
		\begin{align*}
			I\le \frac{C}{Z(\gd)}(2^{-2Nt}+h^{2\gt_\cO r}h^{2(1-\gt_\cO)2r})
			\le \frac{C}{Z(\gd)} (2^{-Nt}+h^{(2-\gt_\cO)r})^2.
		\end{align*}
		
		The second term is bounded by
		\begin{align*}
			II
			&= 
			\left(Z(\gd)^{-1/2}-Z_{N,h}(\gd)^{-1/2}\right)^2Z_{N,h}(\gd)\\
			&\le \frac{Z_{N,h}(\gd)}{\min(Z(\gd), Z_{N,h}(\gd))^3}
			\left|Z_{N,h}(\gd)-Z(\gd)\right|^2\\
			&\le \frac{\|\rho\|_{L^\infty(\bR^k)}}{\min(Z(\gd), Z_{N,h}(\gd))^3}
			\left|Z_{N,h}(\gd)-Z(\gd)\right|^2.
		\end{align*}
		Similar as for $I$, the Lipschitz continuity of $\rho$ and Theorem~\ref{thm:forward-error} then yield 
		\begin{align*}
			\left|Z_{N,h}(\gd)-Z(\gd)\right|^2
			&=
			\left|\int_\gO
			\rho(\gd - [\cO\circ u_{N,h}](\go))
			-\rho(\gd - [\cO\circ u](\go))d\bP(\go)\right|^2 \\
			&\le C(2^{-Nt}+h^{(2-\gt_\cO)r})^2.
		\end{align*}
		Thus, the claim follows by Lemmas~\ref{lem:local-boundedness} and~\ref{lem:lower-bound-Z-disc}, and with Corollary~\ref{cor:lower-bound-Z} since
		\begin{align*}
			d_{\rm Hell}(\bP_\gd, \bP_{\gd, N, h})
			\le 
			C\frac{\max(1,\|\rho\|_{L^\infty(\bR^k)})}{\min(1, Z(\gd), Z_{N,h}(\gd))^{3/2}}
			(2^{-Nt}+h^{(2-\gt_\cO)r})
			\le 
			\frac{C}{c(\gl)^{3/2}}(2^{-Nt}+h^{(2-\gt_\cO)r}).
		\end{align*}
	\end{proof}
	
	Let $\bE_{\gd,N,h}(\cdot):=\bE_{\bP_{\gd,N,h}}(\cdot)$ denote the expectation with respect to the approximated posterior $\bP_{\gd, N, h}$.	
	The bound in Proposition~\ref{prop:posterior_approx} controls the difference of $\bE_\gd$ and $\bE_{\gd,N,h}$.
	\begin{thm}\label{thm:expectation-error}
		Let $(\cX, \left\|\cdot\right\|_\cX)$ be an arbitrary Banach space and let $\varphi\in L^2(\gO, \bP; \cX)$. 
		Under the assumptions of Theorem~\ref{thm:forward-error}, there is for any $\gl>0$  a constant $C(\gl)>0$, independent of $\varphi$, $N$ and $h$, such that  
		\begin{equation*}
			\|\bE_{\gd}(\varphi) - \bE_{\gd, N, h}(\varphi)\|_\cX
			\le C(\gl)\|\varphi\|_{L^2(\gO, \bP; \cX)}(2^{-Nt}+h^{(2-\gt_\cO)r})
			\quad \text{for all $\gd\in B_\gl(0)$.}
		\end{equation*}
	\end{thm}
	
	\begin{proof}
		We fix $\gl>0$ and arbitrary $\gd\in B_\gl(0)$.
		Clearly, the density $\rho$ is continuous and bounded on $\bR^k$, hence $\|\rho\|_{L^\infty(\bR^k)}<\infty$.
		This together with Corollary~\ref{cor:lower-bound-Z} and $\varphi\in L^2(\gO, \bP; \cX)$ then shows
		\begin{align*}
			\|\varphi\|_{L^2(\gO, \bP_\gd; \cX)}^2
			&=
			\int_\gO \|\varphi(\go)\|_\cX^2 \frac{\exp\left(-\Phi(\go; \gd)\right)}{Z(\gd)}d\bP(\go) \\
			&=
			\frac{1}{Z(\gd)}\int_\gO \|\varphi(\go)\|_\cX^2 \rho(\gd - \cG(\go))d\bP(\go) \\
			&\le \frac{\|\rho\|_{L^\infty(\cD)}}{c(\gl)} \|\varphi\|_{L^2(\gO, \bP; \cX)}^2 
			<\infty.
		\end{align*}
		Thus, $\varphi\in L^2(\gO, \bP_\gd; \cX)$ is bounded uniformly for $\gd\in B_\gl(0)$.
		We find in the same fashion that $\varphi\in L^2(\gO, \bP_{\gd, N, h}; \cX)$ is bounded uniformly in $N$, $h$ and for $\gd\in B_\gl(0)$, as Lemma~\ref{lem:lower-bound-Z-disc} shows that
		\begin{align*}
			\|\varphi\|_{L^2(\gO, \bP_{\gd, N, h}; \cX)}^2
			&=
			\frac{1}{Z_{N,h}(\gd)}\int_\gO \|\varphi(\go)\|_\cX^2 \rho(\gd - [\cO\circ u_{N,h}](\go))d\bP(\go) \\
			&\le \frac{\|\rho\|_{L^\infty(\cD)}}{c(\gl)} \|\varphi\|_{L^2(\gO, \bP; \cX)}^2 
			<\infty.
		\end{align*}
		By \cite[Lemma 6.37]{stuart2010inverse} we then obtain 
		\begin{align*}
			\|\bE_{\gd}(\varphi) - \bE_{\gd, N, h}(\varphi)\|_\cX
			\le 
			2 \left(2\frac{\|\rho\|_{L^\infty(\cD)}}{c(\gl)}\right)^{1/2}
			\|\varphi\|_{L^2(\gO, \bP; \cX)}
			d_{\rm Hell}(\bP_\gd, \bP_{\gd, N, h}),
		\end{align*}
		and the claim follows for $C(\gl):=2 \left(2\frac{\|\rho\|_{L^\infty(\cD)}}{c(\gl)}\right)^{1/2}$ by Proposition~\ref{prop:posterior_approx}\,.
	\end{proof}
	
	\section{Markov Chain \MC}
	\label{sec:MCMC}
	
	We use Markov chain \MC\, (MCMC) sampling for the (approximate) posterior measure $\bP_{\gd, N, h}$, 
        where we assume that $\gd\in B_\gl(0)$ for a \emph{fixed} $\gl>0$. 
        For a concise notation, we equilibrate truncation and FE error by assuming $h^{(2-\gt_\cO)r}\simeq 2^{-Nt}$,
        and use the abbreviations $\bP_h:=\bP_{\gd, N, h}$ and $\bE_h:=\bE_{\gd, N, h}$ throughout.

	\subsection{Singlelevel Markov chain \MC} 
        \label{sec:SLMCMC}
	Given the current state $\go^{(k)}$, 
        we draw a candidate $v^{(k)}\sim \bQ(\go^{(k)}; \cdot)$, where $\bQ(\go^{(k)};\,\cdot):\cA\to[0,1]$ 
        is a given proposal probability measure on $(\gO,\cA)$, depending on the current state $\go^{(k)}$.
	We further define the measures $\nu$ and $\nu^\top$ on $\gO\times\gO$ via 
        $d\nu(\go,v):=\bQ(\go; d v)d\bP_h(\go)$ and $d\nu(\go,v)^\top:=\bQ(v; d\go)d\bP_h(v)$ for any $(\go,v)\in\gO\times\gO$,
	and suppose that $\bQ$ is chosen such that $\nu^\top \ll \nu$.
	The new proposal is accepted as next state, i.e. $\go^{(k+1)}=v^{(k)}$, 
        with acceptance probability 
	\begin{equation}\label{eq:acceptance}
		\ga(\go^{(k)}, v^{(k)})
		:=\min
		\left\{1, \frac{d\nu^\top(\go^{(k)}; v^{(k)})}{d\nu(\go^{(k)}; v^{(k)})}\right\}.
	\end{equation}
	Note that $\ga$ in~\eqref{eq:acceptance} is well-defined due to the assumption $\nu^\top \ll \nu$. 
	If $v^{(k)}$ is rejected, we keep the current state $\go^{(k+1)}=\go^{(k)}$.
	This approach is a variant of the Metropolis-Hastings algorithm and 
        generates a Markov chain $(\go^{(k)}, k\in\bN)$ with stationary distribution $\bP_h$. 
	Clearly, the generated samples $(\go^{(k)}, k\in\bN)$ are correlated in a non-trivial way. 
	It is well-known that
    a good choice of proposal density $\bQ$ leads to low correlation and an efficient algorithm. 
	We will in particular focus on the \textit{independence sampler}, where 
	$\bQ(\go^{(k)}; dv^{(k)})=d\bP(v^{(k)})$, that is, 
        the proposal $v^{(k)}$ is drawn from the prior measure $\bP$, 
        independent of the current state $\go^{(k)}$ of the Markov chain.
	
	Now let $\go^{(1)},\dots, \go^{(M)}$, $M\in\bN$ be a sequence of MCMC samples from $\bP_h$.
	We aim to estimate the posterior mean of $\varphi:\gO\to\bR,\; \go\mapsto [\Psi\circ u](\go)$, 
        where 
	$\Psi:V\to \bR$ is a deterministic functional and $u$ is the solution to~\eqref{eq:elliptic-pde-besov}.
	The corresponding \emph{\MCMC\, estimator} of $\bE_h(\varphi)$ is then denoted by 
	\begin{equation}\label{eq:mcmc-est}
		E_M^h(\varphi):= \frac{1}{M}\sum_{i=1}^M \varphi(\go^{(i)}).
	\end{equation}
	
	The sampling error of the \MCMC\, estimator is bounded by the next result. 
	
	\begin{lem}[Geometric ergodicity of independence sampler]\label{lem:geom_erg}
		Let $\cQ$ denote the distribution of the initial sample $\go^{(1)}$ and let $\cP_{\cQ}$ be the probability measure on the probability space generated by the \MCMC\, independence sampler. Furthermore, we denote by $\cE_\cQ$ the expectation with respect to $\cP_{\cQ}$.
		There exists $C=C(\gl)>0$ (recall that $\gd\in B_\gl(0)$) such that for all $\varphi\in L^2(\gO, \bP)$ and $h\in\mfH$ there holds 
		\begin{equation}\label{eq:geometric_mse}
			\cE_\cQ\left(\left(\bE_h(\varphi)-E_M^h(\varphi)\right)^2\right)
			\le C \|\varphi\|_{L^2(\gO, \bP)}^2M^{-1}.
		\end{equation}
	\end{lem}
	\begin{proof}
		Lemma~\ref{lem:lower-bound-Z-disc} yields for any $h\in\mfH$ and $N\in\bN$ the uniform lower bound 
		\begin{equation*}
			Z_{N,h}(\gd)\ge c(\gl)>0\quad
			\text{for all $\gd \in B_\gl(0)$},
		\end{equation*}
		and hence 
		\begin{equation*}
			\esssup_{\go\in\gO}\; \frac{\exp(-\Phi_{N,h}(\go; \gd))}{Z_{N,h}(\gd)}
			\le \frac{\|\rho\|_\infty}{c(\gl)}<\infty.
		\end{equation*}
		We assume that $\frac{\|\rho\|_\infty}{c(\gl)}>1$ without loss of generality.
		Let $\bQ^{(n)}(\widetilde \go\,; \cdot):\gO\to [0,1]$ denote the the distribution of the Markov chain after $n$ steps when starting from (a fixed) $\widetilde \go\in\gO$.
		By \cite[Theorem 1 and Eq. (13)]{smith1996exact} it holds that for any $A\in\cA$ that 
		\begin{equation*}
			\left|\int_A \bQ^{(n)}(\widetilde \go, d\go) - \bP_h(A)\right|\le C \left(1-\frac{c(\gl)}{\|\rho\|_\infty}\right)^n, 
		\end{equation*} 
		thus the Markov chain converges geometrically to the target measure (note that $C$ depends on the distribution of the initial sample $\widetilde \go$.)
		The error bound~\eqref{eq:geometric_mse} then follows exactly in the same way as for the log-normal case in \cite[Lemma B.2 (p. 41/42)]{hoang2020analysis}, and is thus omitted here. 
	\end{proof}
	
	\begin{rem}
		\label{rem:proposal-kernel}
		We remark that the subsequent error analysis of the MCMC estimator in~\eqref{eq:mcmc-est} and its multilevel extension in Subsection~\ref{subsec:ml-mcmc} does \emph{not} rely on the independence sampler as proposal kernel. Our results rather extend to any choice of $\bQ$ in~\eqref{eq:acceptance} that satisfies a geometric ergodicity result as in Lemma~\ref{lem:geom_erg}.  
	\end{rem}
	We may only sample from an approximated quantity of interest (QoI) $\Psi\circ u_{N',h'}\approx \varphi$, where $h'>0$ and $N'\in\bN$ are discretization parameters as in Theorem~\ref{thm:forward-error}, that not necessarily need to coincide with $h$ and $N$ from $\bE_h$. We make the following assumptions on $\Psi$ to bound the resulting discretization error.

	\begin{assumption}\label{ass:functional}~
		Let $\Psi:V\to \bR$ and let $u:\gO\to V$ be the pathwise weak solution to~\eqref{eq:elliptic-pde-besov}.
		\begin{enumerate}[1.)]
			\item 
			Let $\gt_\Psi\in [0,1]$, let $\Psi:H^{\gt_\Psi}(\cD)\to \bR$ be a Fréchet-differentiable functional and denote by 
			\begin{equation*}
				\Psi':H^{\gt_\Psi}(\cD)\to \cL(H^{\gt_\Psi}(\cD); \bR)=(H^{\gt_\Psi}(\cD))'
			\end{equation*}
			the Fréchet-derivative of $\Psi$. 
			There are constants $C>0$, $\rho_1,\rho_2\ge 0$ such that for all $v\in H^{\gt_\Psi}(\cD)$  
			\begin{equation}\label{eq:functionalgrowth}
				|\Psi(v)|\le C(1+\|v\|_{H^{\gt_\Psi}(\cD)}^{\rho_1}), \quad
				\|\Psi'(v)\|_{\cL(H^{\gt_\Psi}(\cD); \bR)}\le C(1+\|v\|_{H^{\gt_\Psi}(\cD)}^{\rho_2}).
			\end{equation}
			\item \label{item:Lqregularity} There holds $u\in L^{q}(\gO;V)$ and there are constants $t>0$, $r\in(0,1]$, such that Theorem~\ref{thm:forward-error} holds for $q:=6\max(\rho_1,\rho_2+1)$.
		\end{enumerate}
	\end{assumption}
	\begin{rem}
		Assumption~\ref{ass:functional} is natural, and includes in particular bounded linear functionals $\Psi$, where $\rho_1=1$ and $\rho_2=0$. Moreover, the condition $q=6\max(\rho_1,\rho_2 + 1)\ge 6$ in the second part is necessary to to bound the MSE in Theorem~\ref{thm:ml-mcmc} below. However, this restriction only applies in case that $p=1$, since $q$ may be arbitrary large in Theorems~\ref{thm:solution-regularity} and~\ref{thm:forward-error} for $p>1$.
	\end{rem}
	
	We record the following result to bound the approximation error. 
	
	\begin{thm}\label{thm:qoi-error}
		Under Assumption~\ref{ass:functional}, there is a constant $C>0$, such that for any $N'\in\bN$ and $h'\in\mfH$ there holds 
		\begin{equation}
			\label{eq:qoi-bound}
			\|\Psi(u)-\Psi(u_{N',h'})\|_{L^6(\gO)} 
			\le C\left(2^{-tN'} + h'^{(2-\gt_\Psi)r}\right)
		\end{equation}
	\end{thm}
	
	\begin{proof}
		The claim is shown in the second part of the proof of \cite[Theorem 5.4]{SS22}.
	\end{proof}
	\begin{rem}
		The estimate~\eqref{eq:qoi-bound} implies in particular the weaker bound 
		\begin{equation*}
			\|\Psi(u)-\Psi(u_{N',h'})\|_{L^2(\gO)} 
			\le C\left(2^{-tN'} + h'^{(2-\gt_\Psi)r}\right), 
		\end{equation*}
		which will be used in Corollary~\ref{cor:mcmc-error} below. 
		On the other hand, we require the bound with respect to $L^6(\gO)$ in~\eqref{eq:qoi-bound} to bound the mean-squared error of the multilevel MCMC estimator in Theorem~\ref{thm:ml-mcmc} in the next subsection.
	\end{rem}
	
	Based on~\eqref{eq:qoi-bound}, 
	we assume that $2^{-tN'} \simeq h'^{(2-\gt_\Psi)r}$ for simplicity and 
	consider the approximated QoI $\varphi_{h'}:=\Psi(u_{N',h'})$.
	The overall error of the \MCMC\, estimator then depends 
	on the regularity of the functionals $\cO$ and $\Psi$. 
	For notational convenience we introduce the variables
	\begin{equation}
		\eta_\cO:=2-\gt_\cO\in[1,\frac{3}{2}),
		\quad\text{and}\quad
		\eta_\Psi:=2-\gt_\Psi\in[1,2].
	\end{equation}

	\begin{cor}
		\label{cor:mcmc-error}
		Under Assumption~\ref{ass:functional}, there is a $C>0$, independent of $h$ and $h'$, such that
		\begin{equation*}
			\cE_\cQ\left(\left(\bE_\gd(\varphi)-E_M^h(\varphi_{h'})\right)^2\right)
			\le C\left(h^{2\eta_\cO r}+(h')^{2\eta_\Psi r}+M^{-1}\right).
		\end{equation*}
	\end{cor}
	
	\begin{proof}
		We consider the error splitting
		\begin{align*}
			\cE_\cQ\left(\left(\bE_\gd(\varphi)-E_M^h(\varphi_{h'})\right)^2\right)
			&\le 3 \cE_\cQ\left(\left(\bE_\gd(\varphi)-\bE_h(\varphi)\right)^2\right) \\
			&\quad +3\cE_\cQ\left(\left(\bE_h(\varphi)-\bE_h(\varphi_{h'})\right)^2\right) \\
			&\quad +3\cE_\cQ\left(\left(\bE_h(\varphi_{h'})-E_M^h(\varphi_{h'})\right)^2\right)\\
			&:= 3(I+II+III).
		\end{align*}
		As $\varphi\in L^2(\gO, \bP)$ and $2^{-Nt}\simeq h^{\eta_\cO r}$, the term $I$ is bounded by Theorem~\ref{thm:expectation-error} via
		\bee
		I= \left(\bE_\gd(\varphi)-\bE_h(\varphi)\right)^2
		= C(\gl)\|\varphi\|_{L^2(\gO, \bP; \cX)}(2^{-2Nt}+h^{2\eta_\cO r}) 
		\le C h^{2\eta_\cO r}.
		\eee
		We use $2^{-N't}\simeq (h')^{\eta_\Psi r}$, Theorems~\ref{thm:expectation-error} and~\ref{thm:qoi-error} to bound the term $II$ via
		\begin{align*}
			II 
			&=\left(\bE(\varphi-\varphi_{h'})
			+\bE_h(\varphi-\varphi_{h'})
			-\bE(\varphi-\varphi_{h'})
			\right)^2 \\
			&\le2\bE(|\varphi-\varphi_{h'}|^2)
			+2 \left(\bE_h(\varphi-\varphi_{h'})
			-\bE(\varphi-\varphi_{h'})
			\right)^2\\
			&\le C\bE(|\varphi-\varphi_{h'}|^2)(1+ h^{2\eta_\cO r})\\
			&\le C (h')^{2\eta_\Psi r}.
		\end{align*}
		The claim now follows with Lemma~\ref{lem:geom_erg}, as the term $III$ is bounded by 
		\begin{equation}
			III=\cE_\cQ\left(\left(\bE_h(\varphi_{h'})-E_M^h(\varphi_{h'})\right)^2\right)
			\le C \|\varphi_{h'}\|_{L^2(\gO, \bP)}^2M^{-1}
			\le C \|\varphi\|_{L^2(\gO, \bP)}^2M^{-1}.
		\end{equation}
	\end{proof}

	\subsection{\MML Markov chain \MC}
	\label{subsec:ml-mcmc}
	
	Now let $L\in\bN$ and consider refining sequences $h_0>\dots>h_L$ and $N_0< \dots < N_L$ of approximation parameters.
	We further denote the approximated posterior on level $\ell=0,\dots, L$ by $\bE_\ell:=\bE_{\gd, N_\ell, h_\ell}$. 
	Given a fixed $\ell$, we choose $L'(\ell)\in\bN$ and let the approximated QoI on levels $\ell'=0,\dots, L'(\ell)$ be given by $\varphi_{\ell'}:=\Psi(u_{N_{\ell'}',h_{\ell'}'})$. 
	We then approximate $\bE_\gd(\varphi)$ using telescopic sums via
	\begin{equation}\label{eq:ml-telescopic}
		\begin{split}
			\bE_\gd(\varphi)\approx \bE_L(\varphi)
			&= \sum_{\ell=1}^L \bE_\ell(\varphi)-\bE_{\ell-1}(\varphi) + \bE_{0}(\varphi) \\
			&\approx \sum_{\ell=1}^L \bE_\ell(\varphi_{L'(\ell)})-\bE_{\ell-1}(\varphi_{L'(\ell)}) + \bE_{0}(\varphi_{L'(0)}) \\
			&= \sum_{\ell=1}^L \sum_{\ell'=1}^{L'(\ell)} \bE_\ell(\varphi_{\ell'}-\varphi_{\ell'-1})-\bE_{\ell-1}(\varphi_{\ell'}-\varphi_{\ell'-1}) 
			+ \sum_{\ell=1}^{L} \bE_\ell(\varphi_0)-\bE_{\ell-1}(\varphi_0) \\
			&\quad+ \sum_{\ell'=1}^{L'(0)} \bE_{0}(\varphi_{\ell'}-\varphi_{\ell'-1}) +  \bE_{0}(\varphi_0).
		\end{split}
	\end{equation} 
	We further introduce the truncation function 
	\begin{equation}\label{eq:ml-truncation}
		\cI_\ell(\go):= \indi_{\{ \Phi_\ell(\go, \gd) \le  \Phi_{\ell-1}(\go, \gd)\}}\in\{0,1\},
		\quad \ell=1,\dots, L,
	\end{equation}
	where $\Phi_\ell(\go, \gd)$ is the level $\ell$-approximation of the Bayesian potential, i.e, 
	\begin{equation}\label{eq:potential-approx}
		\Phi_\ell(\go, \gd):=-\log\left(\rho(\gd-[\cO\circ u_{N_\ell,h_\ell}](\go))\right).
	\end{equation}
	Following \cite[Section 4 and Proposition A.1]{hoang2020analysis}, this allows us to represent the expansion in~\eqref{eq:ml-telescopic} via 
	\begin{equation}\label{eq:ml-A}
		\begin{split}
			&\sum_{\ell=1}^L \bE_\ell(\varphi_{L'(\ell)})-\bE_{\ell-1}(\varphi_{L'(\ell)}) + \bE_{0}(\varphi_{L'(0)})\\
			= &\sum_{\ell=1}^L \sum_{\ell'=0}^{L'(\ell)}
			\bE_\ell(A^{(1)}_{\ell, \ell'}) + \bE_{\ell-1}(A^{(2)}_{\ell, \ell'})
			+
			\bE_\ell(A^{(3)}_{\ell}) \bE_{\ell-1}(A^{(4)}_{\ell, \ell'}+A^{(8)}_{\ell, \ell'})\\
%			\bE_{\ell-1}(A^{(5)}_{\ell}) \bE_{\ell}(A^{(6)}_{\ell, \ell'}+A^{(7)}_{\ell, \ell'}) \\
			&+
			 \sum_{\ell=1}^L \sum_{\ell'=0}^{L'(\ell)}
			 \bE_{\ell-1}(A^{(5)}_{\ell}) \bE_{\ell}(A^{(6)}_{\ell, \ell'}+A^{(7)}_{\ell, \ell'}) 
			 +
			 \sum_{\ell'=1}^{L'(\ell)} \bE_{0}(\varphi_{\ell'}-\varphi_{\ell'-1}) +  \bE_{0}(\varphi_0),
		\end{split}
	\end{equation} 
	where the terms $A^{(1)}-A^{(8)}$ are given for $\varphi_{-1}:=0$  by 
	\begin{equation}\label{eq:A-def}
		\begin{split}
			A^{(1)}_{\ell, \ell'}&:=(1-\exp(\Phi_\ell(\go, \gd) -  \Phi_{\ell-1}(\go, \gd)))
			(\varphi_{\ell'}-\varphi_{\ell'-1})\cI_\ell, \\
			A^{(2)}_{\ell, \ell'}&:=(\exp(\Phi_{\ell-1}(\go, \gd) -  \Phi_{\ell}(\go, \gd))-1)
			(\varphi_{\ell'}-\varphi_{\ell'-1})(1-\cI_\ell), \\
			A^{(3)}_{\ell}&:=(\exp(\Phi_\ell(\go, \gd) -  \Phi_{\ell-1}(\go, \gd))-1)
			\cI_\ell, \\
			A^{(4)}_{\ell, \ell'}&:=(\varphi_{\ell'}-\varphi_{\ell'-1})\cI_\ell, \\
			A^{(5)}_{\ell}&:=(1-\exp(\Phi_{\ell-1}(\go, \gd) -  \Phi_{\ell}(\go, \gd)))
			(1-\cI_\ell), \\
			A^{(6)}_{\ell, \ell'}&:=\exp(\Phi_\ell(\go, \gd) -  \Phi_{\ell-1}(\go, \gd))
			(\varphi_{\ell'}-\varphi_{\ell'-1})\cI_\ell, \\
			A^{(7)}_{\ell, \ell'}&:=(\varphi_{\ell'}-\varphi_{\ell'-1})(1-\cI_\ell), \\
			A^{(8)}_{\ell, \ell'}&:=\exp(\Phi_{\ell-1}(\go, \gd) -  \Phi_{\ell}(\go, \gd))
			(\varphi_{\ell'}-\varphi_{\ell'-1})(1-\cI_\ell).
		\end{split}
	\end{equation} 
	We now replace the expectations $\bE_\ell$ in~\eqref{eq:ml-A} by \MCMC\, estimators $E_{M_{\ell, \ell'}}^\ell(\cdot)$, where the number of samples $M_{\ell, \ell'}$ depends on both discretization levels $\ell$ and $\ell'$. 
	This yields the \emph{\ML\MCMC} (ML-MCMC) estimator 
	
	\begin{equation}\label{eq:ml-mcmc}
		\begin{split}
			E_L(\varphi)
			&:= \sum_{\ell=1}^L \sum_{\ell'=0}^{L'(\ell)}
			E_{M_{\ell, \ell'}}^\ell(A^{(1)}_{\ell, \ell'}) + E_{M_{\ell, \ell'}}^{\ell-1}(A^{(2)}_{\ell, \ell'})\\
			&\qquad\qquad +E_{M_{\ell, \ell'}}^\ell(A^{(3)}_{\ell}) E_{M_{\ell, \ell'}}^{\ell-1}(A^{(4)}_{\ell, \ell'}+A^{(8)}_{\ell, \ell'})+
			E_{M_{\ell, \ell'}}^{\ell-1}(A^{(5)}_{\ell}) E_{M_{\ell, \ell'}}^\ell(A^{(6)}_{\ell, \ell'}+A^{(7)}_{\ell, \ell'}) \\
			&\quad+ \sum_{\ell'=1}^{L'(\ell)} E_{M_{0, \ell'}}^0(\varphi_{\ell'}-\varphi_{\ell'-1}) 
			+  E_{M_{0, 0}}^0(\varphi_0).
		\end{split}
	\end{equation} 
	
	For technical reasons, we require the following assumption on the ML-MCMC estimator.
	\begin{assumption}
		\label{ass:independence}
		Let $L\in\bN$ be given and fix a discretization level $\ell\in\{0,\dots,L\}$.
		The estimators $(E_{M_{\ell, \ell'}}^\ell(A^{(3)}_{\ell}),\, E_{M_{\ell, \ell'}}^\ell(A^{(6)}_{\ell, \ell'}+A^{(7)}_{\ell, \ell'}))$ in~\eqref{eq:ml-mcmc} are independent of $(E_{M_{\ell, \ell'}}^{\ell-1}(A^{(5)}_{\ell}),\, E_{M_{\ell, \ell'}}^{\ell-1}(A^{(4)}_{\ell, \ell'}+A^{(8)}_{\ell, \ell'}))$.
		Furthermore, all estimators in~\eqref{eq:ml-mcmc} are independent with respect to $\ell\in\{0,\dots, L\}$, meaning the Markov chains for each posterior refinement level $\ell\in\{0,\dots, L\}$ are generated independently.
	\end{assumption}
	
	Assumption~\ref{ass:independence} is necessary to derive the mean-squared error in Theorem~\ref{thm:ml-mcmc}, without strengthening Lemma~\ref{lem:geom_erg} to fourth moments.
	For fixed $\ell=0,\dots, L$ we denote by $\cP_\ell$ the probability measure on the probability space generated by all Markov chains with posterior refinement level $\ell$. The combination of all $L+1$ measures $\cP_\ell$ yields with the second part of Assumption~\ref{ass:independence} the product probability measure $\cP_L^{\rm ML}:=\cP_0\otimes\dots\otimes\cP_L$, and we denote the associated expectation by $\cE_L^{\rm ML}$. With this at hand we are able to quantify the MSE of the ML-MCMC algorithm:

	\begin{thm}\label{thm:ml-mcmc}
		Let Assumptions~\ref{ass:functional} and~\ref{ass:independence} hold, and let $h_0\in\mfH$ denote the coarsest level FE refinement parameter. 
		For any fixed $0<\eps<1$ set 
		\begin{align*}
			&L:=\left\lceil\frac{-\log_2(\eps)}{\eta_\cO r} + \log_2(h_0)\right\rceil, 
			\quad L'(\ell)=L':=\left\lceil L\frac{\eta_\cO}{\eta_\Psi} \right\rceil \quad \text{for $\ell=0,\dots, L$},\quad\text{and} \\
			&h_\ell = h_02^{-\ell},\quad h_{\ell'} = h_02^{-\ell'}, \quad N_\ell = \left\lceil-\log_2(h_\ell)\frac{\eta_\cO r}{t}\right\rceil,
			\quad N_{\ell'} = \left\lceil-\log_2(h_{\ell'})\frac{\eta_\Psi r}{t}\right\rceil, \\
			&\text{for $\ell' = 0,\dots, L'(\ell)$, $\ell=0,\dots, L$.}
		\end{align*}
		Furthermore, set the number of samples $M_{\ell, \ell'}$ on each level as
		\begin{equation*}
			M_{\ell, \ell'}:=
			\begin{cases}
				\ceil{h_L^{-2r\eta_\cO} w_{0,0}}, \quad \quad &\text{for $\ell = \ell' = 0$,} \\
				\ceil{h_L^{-2r\eta_\cO} h_\ell^{2r\eta_\cO} w_{\ell,0}}, 
				\quad &\text{for $\ell = 1,\dots, L$ and $\ell'=0$,} \\
				\ceil{h_L^{-2r\eta_\cO} h_{\ell'}^{2r\eta_\Psi} w_{0,\ell'}}, 
				\quad &\text{for $\ell = 0$ and $\ell'=1,\dots, L'(0)$,} \\
				\ceil{h_L^{-2r\eta_\cO} h_\ell^{2r\eta_\cO} h_{\ell'}^{2r\eta_\Psi} 
					w_{\ell',\ell}}, 
				\quad &\text{for $\ell = 1,\dots, L$ and $\ell'=1,\dots, L'(\ell)$,}
			\end{cases}
		\end{equation*}
		where the weights $w_{\ell, \ell'}>0$ are selected 
		such that
		there is a $C_w>0$, independent of $L$, satisfying
		\begin{equation}\label{eq:ml-weights}
			w_{0,0}^{-1/2}
			+
			\sum_{\ell=1}^{L'(0)}w_{\ell,0}^{-1/2} +
			\sum_{\ell'=1}^{L'(0)}w_{0,\ell'}^{-1/2}
			+\sum_{\ell=1}^L\sum_{\ell'=1}^{L'(\ell)}w_{\ell, \ell'}^{-1/2}\le C_w<\infty.
		\end{equation}
		
		Then, there is a $C>0$, independent of $\eps$, such that  
		\begin{equation*}
			\cE_L^{\rm ML}\left(\left(\bE_\gd(\varphi)-E_L(\varphi)\right)^2\right)^{1/2}
			\le C\eps.
		\end{equation*}
	\end{thm}
	
	We remark that it is always possible to select admissible weights $w_{\ell, \ell'}$ that satisfy the uniform bound in~\eqref{eq:ml-weights}. Appropriate choices of $w_{\ell, \ell'}$ to achieve (quasi-)optimal computational complexity depend on the parameters $r, \eta_\cO, \eta_\Psi$ and $d$, and are given in Theorem~\ref{thm:ml-complexity} below.
	
	\begin{proof}[Proof of Theorem~\ref{thm:ml-mcmc}.]
		We use the error splitting 
		\begin{align*}
			&\cE_L^{\rm ML}\left(\left(\bE_\gd(\varphi)-E_L(\varphi)\right)^2\right)^{1/2} \\
			&\quad\le 
			|\bE_\gd(\varphi)-\bE_L(\varphi)| \\
			&\quad\quad +\left|\bE_L(\varphi)-
			\left(\sum_{\ell=1}^L \bE_\ell(\varphi_{L'(\ell)})-\bE_{\ell-1}(\varphi_{L'(\ell)}) + \bE_{0}(\varphi_{L'(0)})\right)\right| \\
			&\quad\quad +\cE_L^{\rm ML}\left(\left(\sum_{\ell=1}^L \bE_\ell(\varphi_{L'(\ell)})-\bE_{\ell-1}(\varphi_{L'(\ell)}) + \bE_{0}(\varphi_{L'(0)})-E_L(\varphi)\right)^2\right)^{1/2}\\
			&\quad
			=I+II+III.
		\end{align*}
		The first term is bounded with Theorem~\ref{thm:expectation-error} and the choices of $L$ and $N_L$ by 
		\begin{align*}
			I\le C (2^{-N_Lt}+h_L^{\eta_\cO r})
			\le C (h_02^{-L})^{\eta_\cO r}
			\le
			C\eps.
		\end{align*}
		We expand the first term in $II$ and use Theorems~\ref{thm:expectation-error} and~\ref{thm:qoi-error} to obtain the bound
		\begin{align*}
			II
			&=
			\left|
			\sum_{\ell=1}^L \bE_{\ell}(\varphi)-\bE_{\ell-1}(\varphi)- (\bE_\ell(\varphi_{L'(\ell)})-\bE_{\ell-1}(\varphi_{L'(\ell)})) + \bE_{0}(\varphi)-\bE_{0}(\varphi_{L'(0)})\right| \\
			&\le \sum_{\ell=1}^L 
			|\bE_{\ell}(\varphi-\varphi_{L'(\ell)})
			-\bE_{\ell-1}(\varphi-\varphi_{L'(\ell)})|
			+|\bE_{0}(\varphi)-\bE_{0}(\varphi_{L'(0)})| \\
			&\le C \sum_{\ell=0}^L 
			\|\varphi-\varphi_{L'(\ell)}\|_{L^2(\gO)} h_\ell^{\eta_\cO r} \\
			&\le C \sum_{\ell=0}^L
			h_{L'(\ell)}^{\eta_\Psi r} h_\ell^{\eta_\cO r} \\
			&\le C (h_02^{-L})^{\eta_\cO r}\sum_{\ell=0}^L
			(h_02^{-\ell})^{\eta_\cO r} \\
			&\le C\eps.
		\end{align*}
		
		To bound $III$, we are going to use the representation~\eqref{eq:ml-A} 
        and bound the estimation error with respect to all terms $A_{\ell,\ell'}^{(1)},\dots,A_{\ell,\ell'}^{(8)}$ separately. 
		Let $\cE_\ell$ denote the expectation with respect to $\cP_\ell$, 
        the probability measure on the space generated by the Markov chains on level $\ell$. 
		We have by Taylor-expansion,~\eqref{eq:ml-truncation},~\eqref{eq:potential} and~\eqref{eq:potential-approx}
		\be
		\begin{split}\label{eq:exp-taylor}
			&|(1-\exp(\Phi_\ell(\go, \gd) -  \Phi_{\ell-1}(\go, \gd)))\cI_\ell| \\
			&\le |\Phi_\ell(\go, \gd) -  \Phi_{\ell-1}(\go, \gd)|\\
			&= \frac{1}{2}[\cO\circ (u_{N_\ell, h_\ell}-u_{N_{\ell-1}, h_{\ell-1}})](\go)^\top
			\gS^{-1}(2\gd + [\cO\circ (u_{N_\ell, h_\ell}+u_{N_{\ell-1}, h_{\ell-1}})](\go))\\
			&\le C 
			\|u_{N_\ell,h_\ell}(\go)-u_{N_{\ell-1},h_{\ell-1}}(\go)\|_{H^{\gt_\cO}} 
			(1 +  \|u_{N_\ell,h_\ell}(\go)+u_{N_{\ell-1},h_{\ell-1}}(\go)\|_{H^{\gt_\cO}})
			\\
			&\le C 
			\left(\|u(\go)-u_{N_{\ell-1},h_{\ell-1}}(\go)\|_{H^{\gt_\cO}}
			+	
			\|u(\go)-u_{N_{\ell},h_{\ell}}(\go)\|_{H^{\gt_\cO}}\right) \\
			&\qquad\cdot\left(1 +  \|u_{N_\ell,h_\ell}(\go)\|_{H^{\gt_\cO}}
			+\|u_{N_{\ell-1},h_{\ell-1}}(\go)\|_{H^{\gt_\cO}}\right),	
		\end{split}
		\ee
		where $C=C(\gS, \gl, \cO)$ is independent of $\ell$. 
		Now let us first consider the case $\ell'\ge 1$.
		We obtain by Lemma~\ref{lem:geom_erg}, 
        the estimate in~\eqref{eq:exp-taylor} and
     	Hölder's inequality 
		\begin{align*}
			&\cE_\ell\left(\left(\bE_\ell(A_{\ell,\ell'}^{(1)})
			- E_{M_{\ell, \ell'}}^\ell(A_{\ell,\ell'}^{(1)})\right)^2\right)^{1/2} \\
			&\quad\le C M_{\ell, \ell'}^{-1/2}\|(1-\exp(\Phi_\ell(\cdot, \gd) -  \Phi_{\ell-1}(\cdot, \gd)))
			\cI_\ell(\varphi_{\ell'}-\varphi_{\ell'-1})\|_{L^2(\gO)} \\
			&\quad\le C M_{\ell, \ell'}^{-1/2}
			\| |\Phi_\ell(\cdot, \gd) -  \Phi_{\ell-1}(\cdot, \gd)|
			|\varphi_{\ell'}-\varphi_{\ell'-1}| \|_{L^2(\gO)} \\
			&\quad\le C M_{\ell, \ell'}^{-1/2}
			\left(\|u-u_{N_{\ell-1},h_{\ell-1}}\|_{L^6(\gO; H^{\gt_\cO})}
				+
			\|u-u_{N_{\ell},h_{\ell}}\|_{L^6(\gO; H^{\gt_\cO})}\right)\\
			&\qquad\cdot
				\left(1+\|u_{N_{\ell},h_{\ell}}\|_{L^6(\gO; H^{\gt_\cO})}
				 + \|u_{N_{\ell-1},h_{\ell-1}}\|_{L^6(\gO; H^{\gt_\cO})}\right)
			\|\varphi-\varphi_{\ell'}\|_{L^6(\gO)}. 
		\end{align*}
		As in the proof of Proposition~\ref{prop:posterior_approx}, 
        we then use Theorem~\ref{thm:forward-error} to show that
		\be\label{eq:ml_estimate}
		\begin{split}
			\|u-u_{N_{\ell},h_{\ell}}\|_{L^6(\gO; H^{\gt_\cO})}
			&\le 
			\|u-u_{N_{\ell}}\|_{L^6(\gO; V)}
			+
			\|u_{N_{\ell}}-u_{N_{\ell},h_{\ell}}\|_{L^6(\gO; H^{\gt_\cO})}\\
			&\le 
			C 2^{-N_\ell t} 
			+\|u_{N_{\ell}}-u_{N_{\ell},h_{\ell}}\|_{L^6(\gO; V)}^{\gt_\cO}
			\|u_{N_{\ell}}-u_{N_{\ell},h_{\ell}}\|_{L^6(\gO; H)}^{1-\gt_\cO}\\
			&\le 
			C (2^{-N_\ell t} + h_\ell^{(2-\gt_\cO)r}).
		\end{split}
		\ee
		This shows in particular 
		\begin{equation}\label{eq:ml_estimate2}
			\|u_{N_{\ell},h_{\ell}}\|_{L^6(\gO; H^{\gt_\cO})}
			\le 
			\|u_{N_{\ell},h_{\ell}}-u\|_{L^6(\gO; H^{\gt_\cO})}
			+\|u\|_{L^6(\gO; H^{\gt_\cO})}
			\le 
			C(1+\|u\|_{L^6(\gO; H^{\gt_\cO})}),
		\end{equation}
		where $C=C(N_0, h_0)$, and the last estimate is independent of $N_l$ and $h_\ell$.
		Since $h_{\ell-1} = 2 h_\ell$ and $2^{-N_\ell t} = h_\ell^{(2-\gt_\cO) r} =h_\ell^{\eta_\cO r}$ by the choice of $N_\ell$, the estimates \eqref{eq:ml_estimate},~\eqref{eq:ml_estimate2} and Theorem~\ref{thm:qoi-error} now show
		\begin{align*}
			\cE_\ell\left(\left(\bE_\ell(A_{\ell,\ell'}^{(1)})
			- E_{M_{\ell, \ell'}}^\ell(A_{\ell,\ell'}^{(1)})\right)^2\right)^{1/2}
			\le C M_{\ell, \ell'}^{-1/2} h_\ell^{\eta_\cO r}h_{\ell'}^{\eta_\Psi r}. 
		\end{align*}	
		Similarly, we find for $\ell'\ge1$ and due to $h_{\ell-1} = 2 h_\ell$ that 
		\begin{align*}
			\cE_\ell\left(\left(\bE_{\ell-1}(A_{\ell,\ell'}^{(2)})
			- E_{M_{\ell, \ell'}}^{\ell-1}(A_{\ell,\ell'}^{(2)})\right)^2\right)^{1/2} 
			\le C M_{\ell, \ell'}^{-1/2} h_\ell^{\eta_\cO r}h_{\ell'}^{\eta_\Psi r}.
		\end{align*}
		To treat the error with respect to the third term in~\eqref{eq:ml-A}, we use the triangle inequality to obtain
		\begin{equation}\label{eq:A3-est1}
			\begin{split}
				&\cE_\ell\left(\left(\bE_\ell(A_{\ell}^{(3)})
				\bE_{\ell-1}(A_{\ell,\ell'}^{(4)}+A_{\ell,\ell'}^{(8)})
				- E_{M_{\ell, \ell'}}^\ell(A_{\ell}^{(3)}) 
				E_{M_{\ell, \ell'}}^{\ell-1}(A_{\ell,\ell'}^{(4)}+A_{\ell,\ell'}^{(8)})\right)^2
				\right)^{1/2} \\
				&\le \cE_\ell\left(
				\left( \bE_\ell(A_{\ell}^{(3)}) - 
				E_{M_{\ell, \ell'}}^\ell(A_{\ell}^{(3)})\right)^2
				\bE_{\ell-1}(A_{\ell,\ell'}^{(4)}+A_{\ell,\ell'}^{(8)})^2
				\right)^{1/2}\\
				&\quad+
				\cE_\ell\left( E_{M_{\ell, \ell'}}^\ell(A_{\ell}^{(3)})^2
				\left(
				\bE_{\ell-1}(A_{\ell,\ell'}^{(4)}+A_{\ell,\ell'}^{(8)})
				-
				E_{M_{\ell, \ell'}}^{\ell-1}(A_{\ell,\ell'}^{(4)}+A_{\ell,\ell'}^{(8)})\right)^2
				\right)^{1/2} \\
				&\le 
				|\bE_{\ell-1}(A_{\ell,\ell'}^{(4)}+A_{\ell,\ell'}^{(8)})|
				\cE_\ell\left(
				\left( \bE_\ell(A_{\ell}^{(3)}) - 
				E_{M_{\ell, \ell'}}^\ell(A_{\ell}^{(3)})\right)^2
				\right)^{1/2}\\
				&\quad+
				\sqrt{2}|\bE_\ell(A_{\ell}^{(3)})|
				\cE_\ell\left(
				\left(
				\bE_{\ell-1}(A_{\ell,\ell'}^{(4)}+A_{\ell,\ell'}^{(8)})
				-
				E_{M_{\ell, \ell'}}^{\ell-1}(A_{\ell,\ell'}^{(4)}+A_{\ell,\ell'}^{(8)})\right)^2
				\right)^{1/2}	\\
				&\quad+
				\sqrt{2}\cE_\ell\left( 
				\left(
				E_{M_{\ell, \ell'}}^\ell(A_{\ell}^{(3)})-\bE_\ell(A_{\ell}^{(3)})
				\right)^2 \right)^{1/2} \\
				&\qquad\qquad\cdot\cE_\ell\left(\left( 
				\bE_{\ell-1}(A_{\ell,\ell'}^{(4)}+A_{\ell,\ell'}^{(8)})
				-
				E_{M_{\ell, \ell'}}^{\ell-1}(A_{\ell,\ell'}^{(4)}+A_{\ell,\ell'}^{(8)})\right)^2
				\right)^{1/2}.
			\end{split}
		\end{equation}
		The last line follows from the basic inequality $(c_1+c_2)^2\le 2(c_1^2+c_2^2)$ for $c_1,c_2\in\bR$ together with Assumption~\ref{ass:independence} on the independence of $E_{M_{\ell, \ell'}}^\ell(A^{(3)}_{\ell})$ and $E_{M_{\ell, \ell'}}^{\ell-1}(A^{(4)}_{\ell, \ell'}+A^{(8)}_{\ell, \ell'}))$ with respect to the measure $\cP_\ell$.
		Theorems~\ref{thm:expectation-error} and~\ref{thm:qoi-error} yield for $\ell'\ge1$ with $|\cI_\ell|\le 1$ that 
		\begin{equation}\label{eq:A3-est2}
			|\bE_{\ell-1}(A_{\ell,\ell'}^{(4)}+A_{\ell,\ell'}^{(8)})|
			\le 
			\|A_{\ell,\ell'}^{(4)}+A_{\ell,\ell'}^{(8)}\|_{L^2(\gO)}
			\le Ch_{\ell'}^{\eta_\Psi r}.
		\end{equation}
		As for the bound~\eqref{eq:ml_estimate}, we obtain by \eqref{eq:exp-taylor}, Hölder's inequality and Theorem~\ref{thm:forward-error} that
		\begin{equation}
			\begin{split}\label{eq:A3-est3}
				|\bE_\ell(A_{\ell}^{(3)})|
				\le
				\|A_{\ell}^{(3)}\|_{L^2(\gO)}
				&\le 
				\bE_\ell\left(
				|(\exp(\Phi_\ell(\cdot, \gd) -  \Phi_{\ell-1}(\cdot, \gd))-1)\cI_\ell|^2
				\right)^{1/2} \\
				&\le 
				C
				\left(\|u -  u_{N_{\ell-1},h_{\ell-1}}\|_{L^4(\gO, \bP; H^{\gt_\cO})}
				+
				\|u -  u_{N_\ell,h_\ell}\|_{L^4(\gO, \bP; H^{\gt_\cO})}
				\right)\\
				&\qquad\cdot(1+\|u\|_{L^4(\gO, \bP; H^{\gt_\cO})}) \\
				&\le C h_\ell^{\eta_\cO r}.
			\end{split}
		\end{equation}
		Substituting \eqref{eq:A3-est2} and~\eqref{eq:A3-est3} in \eqref{eq:A3-est1} thus shows
		\begin{align*}
			&\cE_\ell\left(\left(\bE_\ell(A_{\ell}^{(3)})
			\bE_{\ell-1}(A_{\ell,\ell'}^{(4)}+A_{\ell,\ell'}^{(8)})
			- E_{M_{\ell, \ell'}}^\ell(A_{\ell}^{(3)}) 
			E_{M_{\ell, \ell'}}^{\ell-1}(A_{\ell,\ell'}^{(4)}+A_{\ell,\ell'}^{(8)})\right)^2
			\right)^{1/2} \\
			&\le 
			Ch_{\ell'}^{\eta_\Psi r} \cE_\ell\left( \left(\bE_\ell(A_{\ell}^{(3)}) - 
			E_{M_{\ell, \ell'}}^\ell(A_{\ell}^{(3)})\right)^2\right)^{1/2}
			\\
			&\quad+Ch_{\ell}^{\eta_\cO r}\cE_\ell\left(\left(
			\bE_{\ell-1}(A_{\ell,\ell'}^{(4)}+A_{\ell,\ell'}^{(8)})
			-
			E_{M_{\ell, \ell'}}^{\ell-1}(A_{\ell,\ell'}^{(4)}+A_{\ell,\ell'}^{(8)})
			\right)^2\right)^{1/2}\\
			&\quad+\sqrt{2}
			\cE_\ell\left( 
			\left(
			E_{M_{\ell, \ell'}}^\ell(A_{\ell}^{(3)})-\bE_\ell(A_{\ell}^{(3)})
			\right)^2 \right)^{1/2}\\
			&\qquad\qquad\cdot
			\cE_\ell\left(\left( 
			\bE_{\ell-1}(A_{\ell,\ell'}^{(4)}+A_{\ell,\ell'}^{(8)})
			-
			E_{M_{\ell, \ell'}}^{\ell-1}(A_{\ell,\ell'}^{(4)}+A_{\ell,\ell'}^{(8)})\right)^2
			\right)^{1/2}.
		\end{align*}
		We then use once again \eqref{eq:A3-est2},~\eqref{eq:A3-est3} and the same arguments as for the bound on $A_{\ell,\ell'}^{(1)}$ to see that 
		\begin{align*}
			&\cE_\ell\left(
			\left(\bE_\ell(A_{\ell}^{(3)})
			\bE_{\ell-1}(A_{\ell,\ell'}^{(4)}+A_{\ell,\ell'}^{(8)})
			- E_{M_{\ell, \ell'}}^\ell(A_{\ell}^{(3)}) 
			E_{M_{\ell, \ell'}}^{\ell-1}(A_{\ell,\ell'}^{(4)}+A_{\ell,\ell'}^{(8)})\right)^2
			\right)^{1/2} \\
			&\le C\left(
			h_{\ell'}^{\eta_\Psi r} \frac{\|A_{\ell}^{(3)}\|_{L^2(\gO)}}{M_{\ell,\ell'}^{1/2}}
			+
			h_{\ell}^{\eta_\cO r} \frac{\|A_{\ell,\ell'}^{(4)}+A_{\ell,\ell'}^{(8)}\|_{L^2(\gO)}}{M_{\ell,\ell'}^{1/2}}
			+ \frac{\|A_{\ell}^{(3)}\|_{L^2(\gO)}\|A_{\ell,\ell'}^{(4)}+A_{\ell,\ell'}^{(8)}\|_{L^2(\gO)}}{M_{\ell,\ell'}}\right)
			\\
			&\le C\left(
			\frac{h_{\ell}^{\eta_\cO r}h_{\ell'}^{\eta_\Psi r}}{M_{\ell,\ell'}^{1/2}}
			+ \frac{h_{\ell}^{\eta_\cO r}h_{\ell'}^{\eta_\Psi r}}{M_{\ell,\ell'}}\right)\\
			&\le C
			M_{\ell,\ell'}^{-1/2} h_{\ell}^{\eta_\cO r}h_{\ell'}^{\eta_\Psi r},
		\end{align*}
		holds for $\ell'\ge1$, where the last line follows since $M_{\ell,\ell'}\ge 1$. 
		Analogously, we deduce that
		\begin{align*}
			\cE_\ell\left(
			\left(\bE_{\ell-1}(A_{\ell}^{(5)})
			\bE_\ell(A_{\ell,\ell'}^{(6)}+A_{\ell,\ell'}^{(7)})
			- E_{M_{\ell, \ell'}}^{\ell-1}(A_{\ell}^{(5)}) 
			E_{M_{\ell, \ell'}}^\ell(A_{\ell,\ell'}^{(6)}+A_{\ell,\ell'}^{(7)})\right)^2
			\right)^{1/2}  
			\le C
			M_{\ell,\ell'}^{-1/2} h_{\ell}^{\eta_\cO r}h_{\ell'}^{\eta_\Psi r}.
		\end{align*}
		
		For the case that  $\ell'=0$ we lose the factor $h_{\ell'}^{\eta_\Psi r}$ in all estimates for $III$, since $\varphi_{\ell'-1}=0$ by definition. Repeating the previous arguments then yields for $\ell=1,\dots, L$ 
		\begin{align*}
			\cE_\ell\left(\left(\bE_\ell(A_{\ell,0}^{(1)})
			- E_{M_{\ell, 0}}^\ell(A_{\ell,0}^{(1)})\right)^2\right)^{1/2} 
			&\le C M_{\ell, 0}^{-1/2} h_\ell^{\eta_\cO r},
			\\
			\cE_\ell\left(\left(\bE_\ell(A_{\ell,0}^{(2)})
			- E_{M_{\ell, 0}}^\ell(A_{\ell,0}^{(2)})\right)^2\right)^{1/2} 
			&\le C M_{\ell, 0}^{-1/2} h_\ell^{\eta_\cO r},
			\\
			\cE_\ell\left(\bE_\ell(A_{\ell}^{(3)})
			\bE_{\ell-1}(A_{\ell,0}^{(4)}+A_{\ell,0}^{(8)})
			- E_{M_{\ell, 0}}^\ell(A_{\ell}^{(3)}) 
			E_{M_{\ell, 0}}^{\ell-1}(A_{\ell,0}^{(4)}+A_{\ell,0}^{(8)})\right)
			&\le C M_{\ell,0}^{-1/2} h_{\ell}^{\eta_\cO r}, 
			\\
			\cE_\ell\left(\bE_{\ell-1}(A_{\ell}^{(5)})
			\bE_\ell(A_{\ell,0}^{(4)}+A_{\ell,0}^{(8)})
			- E_{M_{\ell, 0}}^{\ell-1}(A_{\ell}^{(5)}) 
			E_{M_{\ell, 0}}^\ell(A_{\ell,0}^{(4)}+A_{\ell,0}^{(8)})\right)
			&\le C M_{\ell,0}^{-1/2} h_{\ell}^{\eta_\cO r}.
		\end{align*}

		Finally, we also have 
		\begin{align*}
			\cE_\ell\left(\left(
			\bE_{0}(\varphi_{\ell'}-\varphi_{\ell'-1})
			- E_{M_{0, \ell'}}^0(\varphi_{\ell'}-\varphi_{\ell'-1})
			\right)^2\right)^{1/2} 
			\le C M_{0,\ell'}^{-1/2} h_{\ell'}^{\eta_\Psi r},
		\end{align*}
		as well as 
		\begin{align*}
			\cE_\ell\left(\left(
			\bE_{0}(\varphi_{0})
			- E_{M_{0, 0}}^0(\varphi_0)\right)^2\right)^{1/2} 
			\le 
			C M_{0, 0}^{-1/2}.
		\end{align*}
		We now collect all estimates and use~\eqref{eq:ml-A} and~\eqref{eq:ml-weights} to bound $III$ by
		\begin{align*}
			III
			&
			\le C \left(
			M_{0, 0}^{-1/2} + \sum_{\ell'=1}^{L'(0)} M_{0,\ell'}^{-1/2} h_{\ell'}^{\eta_\Psi r}
			+ \sum_{\ell=1}^{L} M_{\ell,0}^{-1/2} h_{\ell}^{\eta_\cO r}
			+ \sum_{\ell=0}^L \sum_{\ell'=1}^{L'(\ell)} M_{\ell,\ell'}^{-1/2} h_{\ell}^{\eta_\cO r}h_{\ell'}^{\eta_\Psi r}
			\right) \\
			& 
			\le C h_L^{\eta_\cO r} 
			\left(w_{0,0}^{-1/2} +
			\sum_{\ell'=1}^{L'(0)} w_{0, \ell'}^{-1/2}
			+\sum_{\ell=1}^{L} w_{\ell, 0}^{-1/2}
			+\sum_{\ell=1}^L\sum_{\ell'=1}^{L'(\ell)} w_{\ell, \ell'}^{-1/2}\right)  \\
			&\le C \eps.
		\end{align*}
		
	\end{proof}
	
	We need another assumption on the sampling cost to derive complexity estimates for the ML-MCMC estimator. 
	\begin{assumption}\label{ass:cost}
		One sample of $\Phi_\ell(\cdot;\gd)=-\log(\rho(\gd-u_{N_\ell,h_\ell}))$ and
		$\varphi_\ell=\Psi(u_{N_\ell,h_\ell})$ with $u_{N_\ell,h_\ell}\in V_{h_\ell}$ 
		and $n_\ell:=\dim(V_{h_\ell})=\cO(h_\ell^{-d})$ is realized in $\cO(n_\ell)$ work and memory.
	\end{assumption}

	\begin{thm}
		\label{thm:ml-complexity}
		Let Assumptions~\ref{ass:functional},~\ref{ass:independence} and~\ref{ass:cost} hold. For any given $\eps>0$, there are ML-MCMC parameters  $L, L', h_\ell, N_\ell, M_{\ell, \ell'}$ such that the ML-MCMC estimator satisfies
		\begin{equation}\label{eq:rmse}
			\cE_L^{\rm ML}\left(\left(\bE_\gd(\varphi)-E_L(\varphi)\right)^2\right)^{1/2}
			\le C\eps,
		\end{equation}
		with computational cost $\cC_{MLMC}$ for $\eps\to 0$ of order
		\begin{equation}
			\label{eq:mlmc-cost}
			\cC_{MLMC} =
			\begin{cases}
				\cO(\eps^{-2})
				\quad&\text{if $2r\min(\eta_\cO, \eta_\Psi)>d$}, \\
				\cO(\eps^{-2}|\log_2(\eps)|^3) 
				\quad&
				\text{if $2r\max(\eta_\cO, \eta_\Psi)> d$ and $2r\min(\eta_\cO, \eta_\Psi)=d$}, \\
				\cO(\eps^{-2}|\log_2(\eps)|^5) 
				\quad&\text{if $2r\eta_\cO=2r\eta_\Psi=d$}, \\
				\cO(\eps^{-d/(r\min(\eta_\cO, \eta_\Psi))}) 
				\quad&
				\text{if $2r\max(\eta_\cO, \eta_\Psi)=d$ and $2r\min(\eta_\cO, \eta_\Psi)<d$,} \\
				\cO(\eps^{-d/(r\min(\eta_\cO, \eta_\Psi))-\epsilon}) \quad&
				\text{if $2r\max(\eta_\cO, \eta_\Psi)<d$}.
			\end{cases}
		\end{equation}
		The last complexity estimate for $2r\max(\eta_\cO, \eta_\Psi)<d$ holds for any $\epsilon>0$. 
	\end{thm}
	
	\begin{rem}
		The first three estimates of $C_{MLMC}$ require that $d\in\{1,2\}$, since $r\in (0,1]$, $\eta_\cO\in[1,\frac{3}{2})$ and $\eta_\Psi\in[1,2]$.
		Further, for the frequently used parameter set $r=\eta_\cO=\eta_\Psi=1$ we recover (essentially) an asymptotic complexity of order $\cO(\eps^{-d})$ if $d\ge 2$, which corresponds to the cost of a single sample with spatial resolution $\eps$.
	\end{rem}
	
	\begin{proof}[Proof of Theorem~\ref{thm:ml-complexity}.]
		For given $\eps>0$, we set the ML-MCMC parameters  $L, L', h_\ell, N_\ell, M_{\ell, \ell'}$ as in Theorem~\ref{thm:ml-mcmc}. 
		The weights $w_{\ell, \ell'}$ in $M_{\ell, \ell'}$ are given by $w_{\ell, \ell'}=w_\ell w_{\ell'}$, where we choose
		\begin{align*}
			&w_{\ell} := 
			\begin{cases}
				(\ell+1)^{\ga_1} \quad&\text{if $2\eta_\cO r>d$,} \\
				1+L^2\indi_{\{\ell>0\}} \quad&\text{if $2\eta_\cO r=d$ and $2\eta_\Psi r\ge d$,} \\
				2^{\ga_2\ell} \quad&\text{if $2\eta_\cO r=d$ and $2\eta_\Psi r<d$,} \\
				2^{(d-2\eta_\cO r)(L-\ell)\ga_3} \quad&\text{if $2\eta_\cO r<d$,} \\
			\end{cases}
			\quad \text{and}\quad \\
			&w_{\ell'} := 
			\begin{cases}
				(\ell'+1)^{\ga_1} \quad&\text{if $2\eta_\Psi r>d$,} \\
				1+(L')^2\indi_{\{\ell>0\}} \quad&\text{if $2\eta_\Psi r=d$ and $2\eta_\cO r\ge d$,} \\
				2^{\ga_2\ell'} \quad&\text{if $2\eta_\Psi r=d$ and $2\eta_\cO r<d$,} \\
				2^{(d-2\eta_\Psi r)(L-\ell')\ga_3} \quad&\text{if $2\eta_\Psi r<d$,} \\
			\end{cases}
		\end{align*}
		for parameters $\ga_1>2$, $\ga_2\in(0,d)$ and $\ga_3\in(0,1)$ to be further specified below. 
		In each scenario the choice of $w_{\ell, \ell'}$ satisfies~\eqref{eq:ml-weights} with a constant $C_w>0$ uniformly in $L$, thus we have that 
		\begin{equation*}
			\cE_L^{\rm ML}\left(\left(\bE_\gd(\varphi)-E_L(\varphi)\right)^2\right)^{1/2}
			\le C\eps,
		\end{equation*}
		by Theorem~\ref{thm:ml-mcmc}, and it remains to bound the computational complexity.

		Under Assumption~\ref{ass:cost} and since $\eps\simeq (h_02^{-L})^{\eta_\cO r}$
		the cost to sample $E_L(\varphi)$ is bounded by
		\be\label{eq:ml-cost}
		\begin{split}
			\cC_{MLMC}
			&\le C\left(M_{0,0}h_0^{-d}+\sum_{\ell=1}^L M_{\ell, 0} h_\ell^{-d} + 
			\sum_{\ell'=1}^{L'(0)} M_{0, \ell'} h_{\ell'}^{-d}
			+\sum_{\ell=1}^L \sum_{\ell'=1}^{L'(\ell)} M_{\ell, \ell'} 
			(h_{\ell}^{-d}+h_{\ell'}^{-d})
			\right)\\
			&\le C\eps^{-2}\left(h_0^{-d}w_{0,0}
			+h_0^{2\eta_\cO r-d}\sum_{\ell=1}^L 2^{(d- 2\eta_\cO r)\ell}w_{\ell} 
			+h_0^{2\eta_\Psi r-d}\sum_{\ell'=1}^{L'(0)} 2^{(d- 2\eta_\Psi r)\ell'}w_{\ell'} \right) \\
			&\quad+
			C\eps^{-2}h_0^{2(\eta_\cO +\eta_\Psi )r-d}
			\sum_{\ell=1}^L2^{(d- 2\eta_\cO r)\ell} w_\ell 
			\sum_{\ell'=1}^{L'(\ell)} 2^{- 2\eta_\Psi r\ell'}w_{\ell'} \\
			&\quad+ 
			C\eps^{-2}h_0^{2(\eta_\cO +\eta_\Psi )r-d}
			\sum_{\ell=1}^L2^{- 2\eta_\cO r\ell} w_\ell 
			\sum_{\ell'=1}^{L'(\ell)} 2^{(d - 2\eta_\Psi r)\ell'}w_{\ell'}\\
			&\le C\eps^{-2}\Bigg(
			1+\sum_{\ell=1}^L 2^{(d- 2\eta_\cO r)\ell}w_{\ell}
			\left(1+\sum_{\ell'=1}^{L} 2^{- 2\eta_\Psi r\ell'}w_{\ell'}\right)\\
			&\qquad\quad\qquad+
			\sum_{\ell'=1}^{L'} 2^{(d- 2\eta_\Psi r)\ell'}w_{\ell'}
			\left(1+\sum_{\ell=1}^{L} 2^{- 2\eta_\cO r\ell}w_{\ell}\right)
			\Bigg).
		\end{split}
		\ee
		Here, the last line follows since $L'(\ell)=L'$ is independent of $\ell$.
		The first sum with respect to $\ell$ is then bounded by   
		\begin{equation*}
			\sum_{\ell=1}^L 2^{(d- 2\eta_\cO r)\ell}w_{\ell}
			= \begin{cases}
				\sum_{\ell=1}^L 2^{(d- 2\eta_\cO r)\ell}(\ell+1)^{\ga_1} 
				\quad\;\, \le C_1 
				\quad&\text{if $2\eta_\cO r>d$,} \\
				\sum_{\ell=1}^L 1+L^2
				\qquad\qquad\qquad\;\;\;
				\le C_2 L^3
				\quad&\text{if $2\eta_\cO r=d$ and $2\eta_\Psi r\ge d$,} \\
				\sum_{\ell=1}^L 2^{\ga_2\ell} 
				\qquad\qquad\qquad\quad\;\;\;
				\le C_3 2^{\ga_2L} 
				\quad&\text{if $2\eta_\cO r=d$ and $2\eta_\Psi r<d,$} \\
				\sum_{\ell=1}^L 2^{(d- 2\eta_\cO r)(\ga_3 L + (1-\ga_3)\ell)} \le C_4 2^{(d- 2\eta_\cO r)L} \quad&\text{if $2\eta_\cO r<d$.} 
			\end{cases}
		\end{equation*}
		The constants $C_1, C_2, C_3, C_4\in(0,\infty)$ are independent of $L$,
		and therefore of $\eps$. 
		Since $L\le C|\log_2(\eps)|$ and $\eps\le C 2^{-L\eta_\cO r}$, there is a $C>0$, independent of $\eps$, such that
		\begin{equation}\label{eq:ml-sum1}
			\sum_{\ell=1}^L 2^{(d- 2\eta_\cO r)\ell}w_{\ell}
			\le  \begin{cases}
				C\quad&\text{if $2\eta_\cO r>d$,} \\
				C |\log_2(\eps)|^3
				\quad&\text{if $2\eta_\cO r=d$ and $2\eta_\Psi r\ge d$,} \\
				C \eps^{-\ga_2/(\eta_\cO r)}
				\quad&\text{if $2\eta_\cO r=d$ and $2\eta_\Psi r<d$,} \\
				C \eps^{2-d/(\eta_\cO r)} \quad&\text{if $2\eta_\cO r<d$.} 
			\end{cases}
		\end{equation}
		Similarly, we conclude by $\eps\simeq  2^{-L\eta_\cO r}\simeq 2^{-L'\eta_\Psi r}$ that there is $C>0$, independent of $\eps$, such that
		\begin{equation}\label{eq:ml-sum2}
			\sum_{\ell'=1}^{L'(\ell)} 2^{(d- 2\eta_\Psi r)\ell'}w_{\ell'}
			\le  \begin{cases}
				C\quad&\text{if $2\eta_\Psi r>d$,} \\
				C |\log_2(\eps)|^3
				\quad&\text{if $2\eta_\Psi r=d$ and $2\eta_\cO r\ge d$,} \\
				C \eps^{-\ga_2/(\eta_\Psi r)}
				\quad&\text{if $2\eta_\Psi r=d$ and $2\eta_\cO r<d$,} \\
				C \eps^{2-d/(\eta_\Psi r)} \quad&\text{if $2\eta_\Psi r<d$.} \\
			\end{cases}
		\end{equation}
		Whenever $2\eta_\cO r=d$ and $2\eta_\Psi r\ge d$, there holds by $L\le C|\log_2(\eps)|$ that
		\begin{equation*}
			\sum_{\ell=1}^L 2^{- 2\eta_\cO r\ell}w_{\ell}
			\le 
			L^2\sum_{\ell=1}^L2^{- d\ell}
			<C|\log_2(\eps)|^2.
		\end{equation*}
		
		Next, in case that $2\eta_\cO r=d$ and $2\eta_\Psi r<d$ we have by $\ga_2<d=2\eta_\cO r$ that
		\begin{equation*}
			\sum_{\ell=1}^L 2^{- 2\eta_\cO r\ell}w_{\ell}
			=
			\sum_{\ell=1}^L2^{(- d+\ga_2)\ell}
			<\infty.
		\end{equation*}
		Moreover, in case that $2\eta_\cO r<d$ holds, we have $w_\ell\le 2^{(d- 2\eta_\cO r)\ga_3L}$ and hence
		\begin{equation*}
			\sum_{\ell=1}^L 2^{- 2\eta_\cO r\ell}w_{\ell}
			\le 
			2^{(d- 2\eta_\cO r)\ga_3L} \sum_{\ell=1}^L2^{-2\eta_\cO r\ell}
			\le C \eps^{\ga_3(2-d/(\eta_\cO r))}.
		\end{equation*}
		Altogether, this shows that there is $C>0$ such that for all $L\ge 1$
		\begin{equation}\label{eq:ml-sum3}
			\sum_{\ell=1}^L 2^{- 2\eta_\cO r\ell}w_{\ell}
			\le  \begin{cases}
				C |\log_2(\eps)|^2
				\quad&\text{if $2\eta_\cO r=d$ and $2\eta_\Psi r\ge d$,} \\
				C  \eps^{\ga_3(2-d/(\eta_\cO r))}
				\quad&\text{if $2\eta_\cO r<d$,} \\
				C
				\quad&\text{otherwise.}
			\end{cases}
		\end{equation}
		As $\ga_2<d$ also holds for $2\eta_\cO r<d$ and $2\eta_\Psi r=d$, one may conclude analogously that
		\begin{equation}\label{eq:ml-sum4}
			\sum_{\ell'=1}^{L'(\ell)} 2^{- 2\eta_\Psi r\ell'}w_{\ell'}
			\le  \begin{cases}
				C |\log_2(\eps)|^2
				\quad&\text{if $2\eta_\Psi r=d$ and $2\eta_\cO r\ge d$,} \\
				C \eps^{\ga_3(2-d/(\eta_\Psi r))}
				\quad&\text{if $2\eta_\Psi r<d$,} \\
				C
				\quad&\text{otherwise.}
			\end{cases}
		\end{equation}
		The first three bounds for $C_{MLMC}$ in~\eqref{eq:mlmc-cost} then follow right away by combining the estimates~\eqref{eq:ml-sum1}--~\eqref{eq:ml-sum4} with \eqref{eq:ml-cost}.
		
		Now we consider the case $2\eta_\cO r=d$ and $d>2\eta_\Psi r$, where we choose
		$$
		\ga_2\in 
		\left(0, \frac{d}{2}\left(\frac{d}{2\eta_\Psi r}-1\right)\right]
		\cap \left(0,d\right)
		$$ 
		and $\ga_3=\frac{1}{2}$ to obtain with $\frac{d}{\eta_\Psi r}>2$ and $d=2{\eta_\cO r}$ that
		\begin{align*}
			C_{MLMC}
			&\le 
			C \eps^{-2}\left(
			1+ \eps^{-\ga_2/(\eta_\cO r)}(1+\eps^{\ga_3(2-d/(\eta_\Psi r))})
			+
			\eps^{2-d/(\eta_\Psi r)}(1+C) 
			\right) \\
			&\le 
			C\left(\eps^{-2-\ga_2/(\eta_\cO r)+\ga_3(2-d/(\eta_\Psi r))}+\eps^{-d/(\eta_\Psi r)}
			\right) \\
			&\le 
			C\left(\eps^{-\ga_22/d-1-d/(2\eta_\Psi r)}+\eps^{-d/(\eta_\Psi r)}
			\right)\\
			&\le 
			C \eps^{-d/(\eta_\Psi r)}.
		\end{align*}
		Analogously, if $2\eta_\cO r<d$ and $2\eta_\Psi r=d$, we let 
		$$
		\ga_2\in 
		\left(0, \frac{d}{2}\left(\frac{d}{2\eta_\cO r}-1\right)\right]
		\cap (0,d), 
		$$ 
		and $\ga_3=\frac{1}{2}$ to obtain that $C_{MLMC}\le C \eps^{-d/(\eta_\cO r)}$. 
		
		In the final case where $\max(2\eta_\cO r, 2\eta_\cO r)<d$ holds, we have that
		\begin{align*}
			C_{MLMC}
			&\le 
			C \eps^{-2}\left(
			1+ \eps^{2-d/(\eta_\cO r)}(1+\eps^{\ga_3(2-d/(\eta_\Psi r))})
			+
			\eps^{2-d/(\eta_\Psi r)}(1+\eps^{\ga_3(2-d/(\eta_\cO r))})
			\right) \\
			&\le 
			C\left(\eps^{\ga_3(2-d/(r\min(\eta_\cO,\eta_\Psi)}
			\right) \eps^{-d/(r\min(\eta_\cO,\eta_\Psi))}\\
			&\le 
			C \eps^{-d/(\min(\eta_\cO,\eta_\Psi)r)-\epsilon},
		\end{align*}
		where  $\epsilon:=-\ga_3(2-d/(r\min(\eta_\cO,\eta_\Psi))>0$ can be made arbitrary small by choosing a sufficiently small $\ga_3\in(0,1)$ (however, note that $C=C(\ga_3)$ in $\eqref{eq:rmse}$ is only uniform in $L$ and $\eps$ if $\ga_3>0$, but we obtain that $C(\ga_3)=\cO(L^2)=\cO(|\log_2(\eps)|^2)$ for $\ga_3=0$.)
	\end{proof}
	
	\section{Numerical Experiments}
	\label{sec:numerics}
	
	\subsection{Bayesian inverse problem in 1D}
	\label{sec:1Dnumerics}
	
	Let $\bT^1=[0,1]$ be the one-dimensional torus, let $\cD:=\bT$, and consider the elliptic (forward) problem to find $u(\go):\cD\to \bR$ for given $\go\in\gO$ such that 
	\begin{equation}\label{eq:ellipticpde-1d}
		\begin{alignedat}{2}
			-\nabla\cdot(\exp(b_T(\go))\nabla u(\go)) &= 10\quad &&\text{in $(0,1)$}, \quad 
			u(\go) = 0 \quad &&\text{on $\{0, 1\}$}.
		\end{alignedat}
	\end{equation}
	The log-diffusion coefficient $b_T$ in~\eqref{eq:ellipticpde-1d} is a Besov random tree prior with parameters $s=\frac{8}{5}$, $p=\frac{5}{3}$ and wavelet density $\beta=\frac{4}{5}$.
	For the Bayesian inverse problem, we sample a realization of $u(\go)$ for given $\go$ (also referred to as "ground truth") and consider the parameter-to-observation map
	\begin{equation*}
		\cG:\gO\to \bR^k,\quad \go\mapsto 
		\begin{pmatrix}
			u(\go, x_i),\, i=1,\dots,k
		\end{pmatrix}^\top,
		\quad\text{for $0<x_1<\dots<x_k<1$,}
	\end{equation*}
	where $k=9$ and $x_i:=0.1\cdot i$ for $i=1,\dots,9$. 
	Hence, the observation functional $\cO$ is a linear functional $\cO\in (V')^k$. 
	To generate the synthetic data, we approximate $u$ on a FE grid on $[0,1]$ with $2^{11}$ equidistant nodes, and by truncating the Besov random tree prior $b_T$ after $N=11$ scales to obtain a feasible log-diffusion $b_{T,N}\approx b_T$ (as $s-\frac{d}{p}=1$ the resulting pathwise error is of order $\cO(2^{-11})$ by Theorem~\ref{thm:forward-error}).
	A plot of the ground truth, the corresponding fine approximation of $u$ and the observations is given on the left panel in Figure~\ref{fig:1d-results}.
	\begin{figure}[ht]
		\centering
		\subfigure{\includegraphics[scale=0.45]{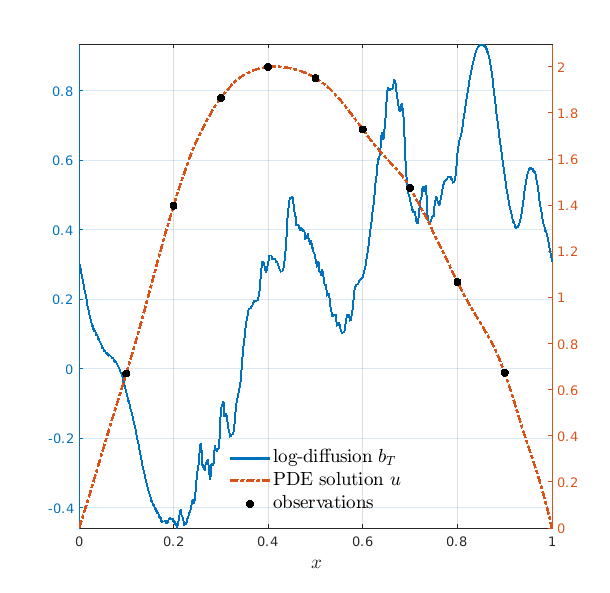}}
		\subfigure{\includegraphics[scale=0.45]{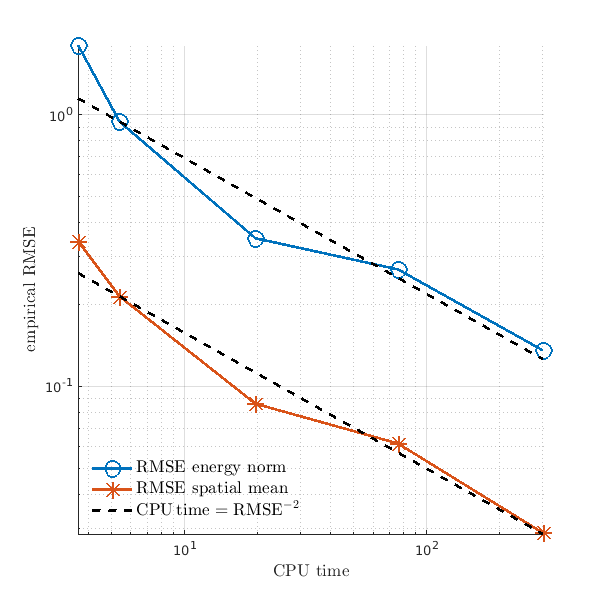}}
		\caption{Left: Plot of the synthetic data, that is, the sampled Besov random tree prior, corresponding PDE solution $u$ and point observations (black dots). 
			Right: Time-to-error plot of ML-MCMC estimator for the energy norm and the spatial mean with maximum refinements $h_L=2^{-3-L}$ for $L=2,\dots, 6$. As predicted, an error of order $\cO(\eps)$ is achieved with computational complexity of order $\cO(\eps^{-2})$.}

		\label{fig:1d-results}
	\end{figure}
	The noisy observations are given by 
	\begin{equation*}
		\gd = \cG(\go) + \vartheta,
	\end{equation*}
	where $\vartheta\sim\cN(0, \sigma^2 I_9)$ with $\sigma=0.1$. 
	
	We aim to approximate the posterior expectations $\bE_\gd(\Psi_i(u))$ for $i=1,2$, where  
	\begin{equation*}
		\Psi_1(u):=\left(\int_\cD \nabla u \cdot \nabla u \,dx\right)^{\frac{1}{2}}
		\quad\text{and}\quad
		\Psi_2(u):=\int_\cD u \,dx
	\end{equation*}
	are the energy norm and spatial mean of $u$, respectively. 
	Assumption~\ref{ass:functional} holds for this QoIs with $\theta_{\Psi_1}=1$, $\theta_{\Psi_2}=0$ and with $\rho_1=1$, $\rho_2=0$ in either case.
	We use the ML-MCMC estimator from Section~\ref{subsec:ml-mcmc} with initial FE mesh width $h_0:=2^{-3}$ and test the cases $L\in\{2,\dots,6\}$. The ML-MCMC parameters are chosen as in Theorem~\ref{thm:ml-mcmc} (for $t=r=s-\frac{d}{p}=1$), where we have used $\eta_\Psi=1$ for both the energy norm and the spatial mean for simplicity.
	Using $\eta_\Psi=2$ for the spatial mean requires a large parameter $\ga$, 
	otherwise we obtain essentially $M_{\ell, \ell'}=1$ for $\ell, \ell'\ge 1$. 
	But if $\ga$ is large, 
	we do not gain a significant reduction in computational time.
	Since $2r\min(\eta_\cO, \eta_\Psi)>d$ for all $\eta_\Psi\in[1,2]$, 
	the asymptotic complexity of order $\cO(\eps^{-2})$ remains unaffected from this simplification.
	We choose the ML-MCMC weights $w_{\ell}=(\ell+1)^\ga, w_{\ell'}=(\ell'+1)^\ga$ with $\ga=3$. According to Theorem~\ref{thm:ml-complexity}, this yields a RMSE of order $\cO(2^{-(L+3)})$ with work $\cO(2^{2L})$ for any $L$.

	We use the single-level MC ratio estimator from \cite[Section 4.1]{scheichl2017quasi} with FE meshwidth $h_{ref}=2^{-11}$, scale truncation $N_{ref}=11$, and $M_{ref}=2^{22}$ samples to obtain a reference solution in our test example. The resulting error of this reference is of order $\cO(2^{-11})$, and therefore negligible when compared to the ML-MCMC estimator with $L\le 6$.  
	We sample $M_{ML}=64$ independent realizations of each ML-MCMC estimator for a given $L$ to calculate the empirical RMSE based on the reference solution. 
	The results are depicted on the right in Figure~\ref{fig:1d-results}. One clearly sees that an empirical error of order $\cO(\eps)$ is achieved in $\cO(\eps^{-2})$ computational time, confirming our theoretical analysis in Section~\ref{subsec:ml-mcmc}.
	
	\subsection{Bayesian inverse problem in 2D}
	\label{sec:2Dnumerics}
	
	Let $\bT^2=[0,1]^2$ be the two-dimensional torus, let $\cD:=\bT^2$, and consider the elliptic (forward) problem to find $u(\go):\cD\to \bR$ for given $\go\in\gO$ such that 
	\begin{equation}\label{eq:ellipticpde-2d}
		\begin{alignedat}{2}
			-\nabla\cdot(\exp(b_T(\go))\nabla u(\go)) &= 10\quad &&\text{in $(0,1)^2$}, \quad 
			u(\go) = 0 \quad &&\text{on $\partial\cD$}.
		\end{alignedat}
	\end{equation}
	The log-diffusion coefficient $b_T$ in~\eqref{eq:ellipticpde-2d} is a Besov random tree prior with parameters $s=\frac{12}{5}$, $p=\frac{5}{3}$ and wavelet density $\beta=\frac{1}{2}$.
	We sample again a realization of $u(\go)$ for a given $\go$ as "ground truth" and now consider the parameter-to-observation map
	\begin{equation*}
		\cG:\gO\to \bR^k,\quad \go\mapsto 
		\begin{pmatrix}
			u(\go, (x_i, y_j)),\, i,j=1,\dots,\sqrt k
		\end{pmatrix}^\top,
	\end{equation*}
	where $k=36$ and with observation points $x_i, y_j \in\{0.1, 0.26, 0.42, 0.58, 0.74, 0.9\}$.
	This yields an observation functional $\cO\notin (V')^k$, but rather $\cO\in ((H^{1+\eps}_0(\cD))')^k$ for any $\eps>0$. However, since $\eps$ may be arbitrary small we treat $\cO$ as if $\eta_\cO=1$ holds in our experiments.
	The synthetic data is sampled by bilinear FEs on an equidistant grid with $2^{10}$ nodes in each coordinate direction, and by truncating the Besov random tree prior $b_T$ after $N=10$ scales. The resulting pathwise error is then of order $\cO(2^{-10})$ by Theorem~\ref{thm:forward-error}.
	A plot of the ground truth, the corresponding fine approximation of $u$ and the observations are given in Figure~\ref{fig:2D-obs}.
	\begin{figure}[ht]
		\centering
		\subfigure{\includegraphics[scale=0.45]{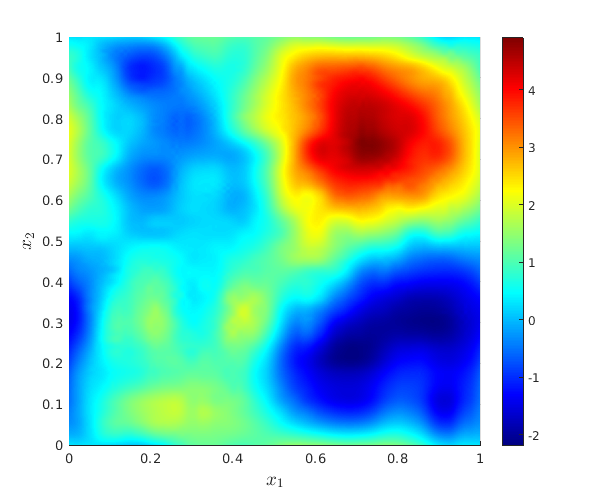}}
		\subfigure{\includegraphics[scale=0.45]{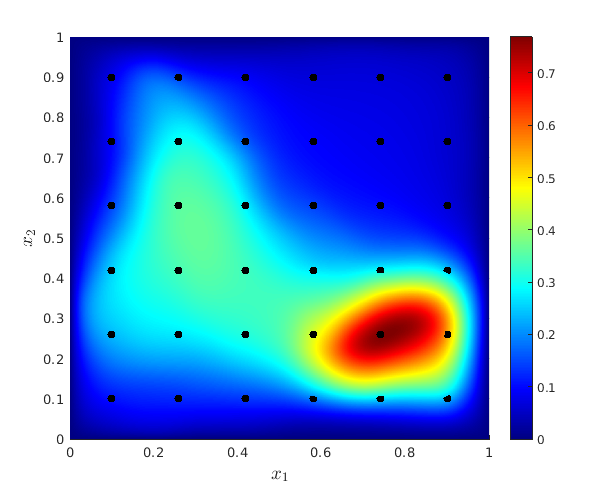}}
		\caption{Left: Plot of the synthetic data, that is, the sampled Besov random tree prior on $\bT^2$ with parameters $s=\frac{12}{5}$, $p=\frac{5}{3}$ and $\beta=\frac{1}{2}$.
			Right:  Corresponding PDE solution $u$ and point observations (black dots).}
		\label{fig:2D-obs}
	\end{figure}
	
	We approximate again the posterior expectations $\bE_\gd(\Psi_i(u))$ for $i=1,2$, where 
	$\Psi_1$ and $\Psi_2$ are the functionals corresponding to the energy norm and spatial mean, respectively.
	We test the ML-MCMC estimator from Section~\ref{subsec:ml-mcmc} with bilinear finite elements, initial mesh width $h_0:=2^{-3}$ and $L\in\{2,\dots,5\}$. The MLMC-parameters are chosen as in Theorem~\ref{thm:ml-complexity} (for $t=r=s-\frac{d}{p}=1$):

	For the energy norm ($\Psi_1$) it holds that $\eta_\Psi=1$, hence $2r\eta_\cO=2r\eta_\Psi=d$ and we set the ML-MCMC weights now as $w_{\ell}=w_{\ell'}=1+3\cdot L^2\indi_{\{\ell>0\}}$. We found that multiplying $w_{\ell}, w_{\ell'}$ by a factor of three for $\ell,\ell'>0$ stabilizes convergence, while this clearly does not affect the asymptotic cost of the estimator.
	For the spatial mean ($\Psi_2$) it holds that $\eta_\Psi=2$, hence $2r\min(\eta_\cO, \eta_\Psi)=d$ and $2r\max(\eta_\cO, \eta_\Psi)>d$, and we exploit the increased smoothness of $\Psi_2$ to reduce computational cost.
	We therefore set the ML-MCMC weights as $w_{\ell}=1+3\cdot L^2\indi_{\{\ell>0\}}$ and $w_{\ell'}=(\ell'+1)^6$.

	For our examples in space dimension $d=2$, 
	all ML-MCMC estimators
	allowed considerable reductions in CPU-time upon allowing 
	a burn-in period of the first Markov chains at each discretization level 
	as follows:
	for a fixed discretization level $\ell$ (of the posterior approximation) and $\ell'=0$ 
	we discarded the first $20 \%$ of samples of the largest Markov chain 
	corresponding to the level $(\ell, 0)$. 
	For $\ell'\ge1$, we then used the last accepted sample 
	of the previous chain on $(\ell, \ell'-1)$ to initialize the new chain 
	with respect to the levels $(\ell, \ell')$, without another burn-in phase.
	We repeat this procedure for all $\ell=0, \dots, L$, 
	where the initial values of the first chains are chosen independently with respect to $\ell$. 
	This modified estimator satisfies in particular Assumption~\ref{ass:independence}.
	To justify our burn-in approach we report the results of the corresponding ML-MCMC estimators without burn-in phase for $\ell'=0$, that initialize the Markov chains for each pair $(\ell, \ell')$ independently.

	We again use the single-level MC ratio estimator from \cite[Section 4.1]{scheichl2017quasi} with FE meshwidth $h_{ref}=2^{-9}$, scale truncation $N_{ref}=9$, and $M_{ref}=2^{18}$ samples to obtain a reference solution for our test example. The resulting error of the reference is now of order $\cO(2^{-9})$, which still seems to be sufficient for our experiments.  
	We sample $M_{ML}=64$ independent realizations of each ML-MCMC estimator for a given $L$ to calculate the empirical RMSE based on the reference solution. 
	The results are depicted in Figure~\ref{fig:2D-results} and Table~\ref{table:RMSE}.
	One clearly sees that an empirical error of order $\cO(\eps)$ is achieved in $\cO(\eps^{-2}|\log(\eps)|^2)$ computational time with the burned-in estimator for both the energy norm and the spatial mean. This is somewhat surprising at first sight, 
	as from our complexity analysis,
	we would expect complexity of order $\cO(\eps^{-2}|\log(\eps)|^5)$ for the energy norm 
	by Theorem~\ref{thm:ml-complexity} and the choice of ML-MCMC weights.
	
	We further see that the initial burn-in phase and the sequential initialization with respect to $\ell'$ significantly reduces the empirical RMSE, while the computational times of both estimators are comparable, see Table~\ref{table:RMSE}. This effect is especially pronounced for the estimator of the spatial mean, which does not seem to converge at all without burn-in.  
	This is explained since very few samples (essentially $M_{\ell, \ell'}=\cO(1)$) are generated for $\ell'\ge1$ in this case. A burn-in phase and initialization of the previous level therefore massively benefits these short chains, while they do not enter the asymptotic realm without burn-in phase on the coarsest level and proper initialization.
	\begin{figure}[ht]
		\centering
		\includegraphics[scale=0.5]{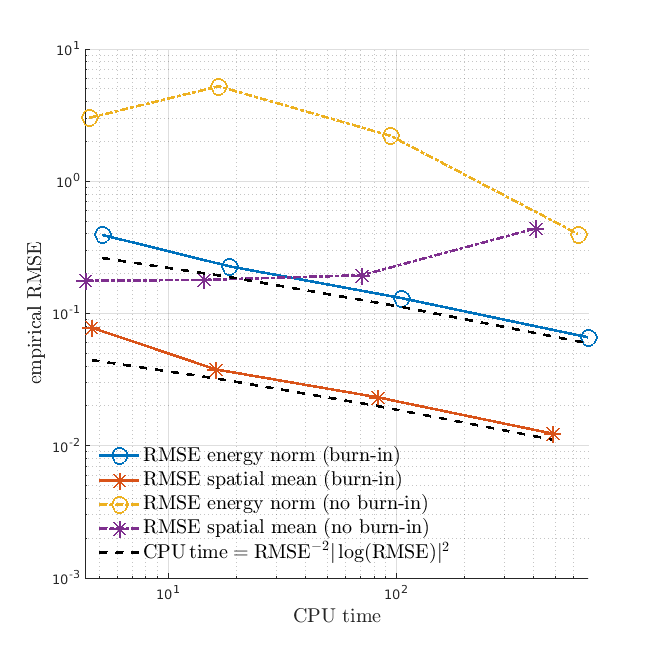}
		\caption{Time-to-error plot of the ML-MCMC estimator for the energy norm (blue circles) and the spatial mean (orange stars). The estimators with burn-in on the coarsest level $\ell'=0$ and sequential initialization achieve an error $\cO(\eps)$ with computational complexity of order $\cO(\eps^{-2}|\log(\eps)|^2)$ and have a significantly lower empirical error than their counterparts without burn-in phase.}
		\label{fig:2D-results}
	\end{figure}

	\begin{table}[ht]
		\centering
		\begin{tabular}{ |c||c|c|c|c| } 
			\hline
			Level $L$ (finest resolution)  & 2 & 3 & 4 & 5 \\ \hline 
			\multirow{2}{8em}{$\frac{\text{RMSE without burn-in}}{\text{RMSE with burn-in}}$} 
			& 7.6692 & 22.9367 & 16.8583 & 6.0254 \\
			& 2.2673 & 4.7504 & 8.3776 & 35.5614  \\ \hline 
			\multirow{2}{10em}{$\frac{\text{CPU time with burn-in}}{\text{CPU time without burn-in}}$} 
			& 1.1372 & 1.1109 & 1.1142 & 1.1045 \\
			& 1.0692 & 1.1341 & 1.0811 & 1.1812  \\ \hline 
		\end{tabular}
		\caption{Ratios of empirical RMSE and CPU time for the ML-MCMC estimators of the energy norm (first row) and the spatial mean (second row). Burn-in and sequential initialization achieve a significant reduction of the RMSE, at an additional cost of less than 20\%.}
		\label{table:RMSE}
	\end{table}
	
	\subsection*{Acknowledgements}
	
	AS was partly funded by the ETH Foundations of Date Science Initiative (ETH-FDS), and it is gratefully acknowledged. The authors would like to thank Prof. Dr. Christoph Schwab for insightful discussions that lead to a significant improvement of the manuscript. 
	
	\bibliographystyle{abbrv}
	\bibliography{references}
\end{document}